\documentclass[12pt]{amsart}       
\usepackage{txfonts}
\usepackage{amssymb}
\usepackage{eucal}
\usepackage{graphicx}
\usepackage{amsmath}
\usepackage{amscd}
\usepackage[all]{xy}           
\usepackage{amsfonts,latexsym}
\usepackage{xspace}
\usepackage{epsfig}
\usepackage{float}
\usepackage{color}
\usepackage{fancybox}
\usepackage{colordvi}
\usepackage{multicol}
\usepackage{colordvi}
\usepackage{ifpdf}
\usepackage{bm}
\usepackage{stmaryrd}
\usepackage[all]{xy}
\ifpdf
\usepackage[colorlinks,final,backref=page,hyperindex]{hyperref}
\else
\usepackage[colorlinks,final,backref=page,hyperindex,hypertex]{hyperref}
\fi
\usepackage[active]{srcltx} 

\usepackage{tikz}

\usepackage{graphicx}

\newtheorem{theorem}{Theorem}[section]
\newtheorem{prop}[theorem]{Proposition}
\newtheorem{lemma}[theorem]{Lemma}
\newtheorem{coro}[theorem]{Corollary}
\newtheorem{prop-def}{Proposition-Definition}[section]
\newtheorem{coro-def}{Corollary-Definition}[section]

\theoremstyle{definition}
\newtheorem{defn}[theorem]{Definition}
\newtheorem{remark}[theorem]{Remark}

\newtheorem{exam}[theorem]{Example}

\newcommand{\nc}{\newcommand}
\nc{\tred}[1]{\textcolor{red}{#1}}
\nc{\tblue}[1]{\textcolor{blue}{#1}}
\nc{\tgreen}[1]{\textcolor{green}{#1}}
\nc{\tpurple}[1]{\textcolor{purple}{#1}}
\nc{\btred}[1]{\textcolor{red}{\bf #1}}
\nc{\btblue}[1]{\textcolor{blue}{\bf #1}}
\nc{\btgreen}[1]{\textcolor{green}{\bf #1}}
\nc{\btpurple}[1]{\textcolor{purple}{\bf #1}}
\nc{\NN}{{\mathbb N}}
\nc{\ncsha}{{\mbox{\cyr X}^{\mathrm NC}}} \nc{\ncshao}{{\mbox{\cyr
			X}^{\mathrm NC}_0}}

\newcommand{\delete}[1]{}

	\nc{\mlabel}[1]{\label{#1}}
	\nc{\mcite}[1]{\cite{#1}}
	\nc{\mref}[1]{\ref{#1}}
	\nc{\meqref}[1]{\eqref{#1}}
	\nc{\mbibitem}[1]{\bibitem{#1}}

\delete{
	\nc{\mlabel}[1]{\label{#1}{\hfill \hspace{1cm}{\bf{{\ }\hfill(#1)}}}}
	\nc{\mcite}[1]{\cite{#1}{{\bf{{\ }(#1)}}}}
	\nc{\mref}[1]{\ref{#1}{{\bf{{\ }(#1)}}}}
	\nc{\meqref}[1]{\eqref{#1}{{\bf{{\ }(#1)}}}}
	\nc{\mbibitem}[1]{\bibitem[\bf #1]{#1}}
}

\nc{\sha}{{\mbox{\cyr X}}}  
\newfont{\scyr}{wncyr10 scaled 550}
\nc{\ssha}{\mbox{\bf \scyr X}}
\nc{\shap}{{\mbox{\cyrs X}}} 
\nc{\shpr}{\diamond}    
\nc{\shp}{\ast} \nc{\shplus}{\shpr^+}
\nc{\shprc}{\shpr_c}    
\nc{\dep}{\mrm{dep}} \nc{\lc}{\lfloor} \nc{\rc}{\rfloor}
\nc{\db}{\leq_{\rm db}} \nc{\bfk}{\bf k}

\nc{\cala}{{\mathcal A}} \nc{\calb}{{\mathcal B}}
\nc{\calc}{{\mathcal C}}
\nc{\cald}{{\mathcal D}} \nc{\cale}{{\mathcal E}}
\nc{\calf}{{\mathcal F}} \nc{\calg}{{\mathcal G}}
\nc{\calh}{{\mathcal H}} \nc{\cali}{{\mathcal I}}
\nc{\call}{{\mathcal L}} \nc{\calm}{{\mathcal M}}
\nc{\caln}{{\mathcal N}} \nc{\calo}{{\mathcal O}}
\nc{\calp}{{\mathcal P}} \nc{\calr}{{\mathcal R}}
\nc{\cals}{{\mathcal S}} \nc{\calt}{{\mathcal T}}
\nc{\calu}{{\mathcal U}} \nc{\calw}{{\mathcal W}} \nc{\calk}{{\mathcal K}}
\nc{\calx}{{\mathcal X}} \nc{\CA}{\mathcal{A}}

\nc{\fraka}{{\mathfrak a}} \nc{\frakA}{{\mathfrak A}}
\nc{\frakb}{{\mathfrak b}} \nc{\frakB}{{\mathfrak B}}
\nc{\frakc}{{\mathfrak c}}
\nc{\frakD}{{\mathfrak D}} \nc{\frakF}{\mathfrak{F}}
\nc{\frakf}{{\mathfrak f}} \nc{\frakg}{{\mathfrak g}}
\nc{\frakH}{{\mathfrak H}} \nc{\frakL}{{\mathfrak L}}
\nc{\frakM}{{\mathfrak M}} \nc{\bfrakM}{\overline{\frakM}}
\nc{\frakm}{{\mathfrak m}} \nc{\frakP}{{\mathfrak P}}
\nc{\frakN}{{\mathfrak N}} \nc{\frakp}{{\mathfrak p}}
\nc{\frakS}{{\mathfrak S}} \nc{\frakT}{\mathfrak{T}}
\nc{\frakX}{{\mathfrak X}}
\nc{\frakZ}{\mathfrak{Z}}
\nc{\frakJ}{\mathfrak{J}}
\nc{\mfrakL}{\Phi}
\nc{\mfrakH}{\Psi}
\nc{\MfrakL}{\phi}
\nc{\MfrakH}{\psi}
\nc{\frakR}{\mathfrak{R}}
\nc{\GL}{\mathrm{GL}}
\nc{\gl}{\mathfrak{gl}}

\nc{\xsj}{\vartriangleright}

\font\cyr=wncyr10 \font\cyrs=wncyr7
\nc{\li}[1]{\textcolor{red}{#1}}
\nc{\lir}[1]{\textcolor{red}{Li:#1}}
\nc{\zong}[1]{\textcolor{blue}{Zong: #1}}
\nc{\xing}[1]{\textcolor{blue}{Xing:#1}}
\nc{\revise}[1]{\textcolor{red}{#1}}

\nc{\dd}{{\rm d}} \nc{\Ad}{{\rm AD}}
\nc{\CC}{\mathbb{C}} \nc{\PP}{\mathbb{P}} \nc{\QQ}{\mathbb{Q}} \nc{\ZZ}{\mathbb{Z}}
\nc{\ZZZ}{\mathbb{Z}^\ast} \nc{\RR}{\mathbb{R}} \nc{\id}{\rm id}
\nc{\AD}{{\rm AD}} \nc{\aad}{{\rm Ad}} \nc{\pown}{P_n} \nc{\powm}{P_m}
\nc{\Aut}{{\rm Aut}} \nc{\ad}{{\rm Ad}} \nc{\fni}{\frac{1}{n}} \nc{\fmi}{\frac{1}{m}}
\nc{\ada}{{\rm ad}} \nc{\mulz}{\cdot_0}

\nc{\complim}{synchronized\xspace}
\nc{\compgroup}{synchronized\xspace}

\nc{\limwtzero}{limit-weight zero\xspace}
\nc{\compint}{synchronized integrable\xspace}
\nc{\mcorr}{exponential adjoint\xspace}
\nc{\rbo}{\mathrm{RBO}}

\nc{\vpa}{\vspace{-.1cm}}
\nc{\vpb}{\vspace{-.2cm}}
\nc{\vpc}{\vspace{-.3cm}}
\nc{\vpd}{\vspace{-.4cm}}
\nc{\vpe}{\vspace{-.5cm}}

\begin{document}

\title[Rota-Baxter groups with weight zero and integration on topological groups]{Rota-Baxter groups with weight zero and integration on topological groups}

%

\author{Xing Gao}
\address{School of Mathematics and Statistics, Lanzhou University
Lanzhou, 730000, China;
Gansu Provincial Research Center for Basic Disciplines of Mathematics
and Statistics
Lanzhou, 730070, China;
School of Mathematics and Statistics
Qinghai Nationalities University, Xining, 810007, China
}
\email{gaoxing@lzu.edu.cn}

\author{Li Guo}
\address{
	Department of Mathematics and Computer Science,
	Rutgers University,
	Newark, NJ 07102, US}
\email{liguo@rutgers.edu}

\author{Zongjian Han}
\address{School of Mathematics and Statistics, Lanzhou University, Lanzhou, Gansu 730000, China}
\email{2715873690@qq.com}

\date{\today}
\begin{abstract}
Rota-Baxter groups with weights $\pm 1$ have attracted quite much attention since their recent introduction, thanks to their connections with Rota-Baxter Lie algebras, factorizations of Lie groups, post- and pre-Lie algebras, braces and set-theoretic solutions of the Yang-Baxter equation. Despite their expected importance from integrals on groups to pre-groups and Yang-Baxter equations, Rota-Baxter groups with weight zero and other weights has been a challenge to define and their search has been the focus of several attempts. 

By composing an operator with a section map as a perturbation device, we first generalize the notion of a Rota-Baxter operator on a group from the existing case of weight $\pm 1$ to the case where the weight is given by a pair of maps and then a sequence limit of such pairs. From there, two candidates of Rota-Baxter operators with  weight zero are given. One of them is the Rota-Baxter operator with limit-weight zero detailed here, with the other candidate introduced in a companion work. This operator is shown to have its tangent map the Rota-Baxter operator with  weight zero on Lie algebras. It also gives concrete applications in integrals of maps with values in a class of topological groups called $\RR$-groups, satisfying a multiplicative version of the integration-by-parts formula.

	
In parallel, differential groups in this framework is also developed and a group formulation of the First Fundamental Theorem of Calculus is obtained.


	
\end{abstract}

\makeatletter
\@namedef{subjclassname@2020}{\textup{2020} Mathematics Subject Classification}
\makeatother
\subjclass[2020]{
22E60, 
17B38, 
17B40, 
16W99, 
45N05 
}

\keywords{Rota-Baxter group with pair-weight; Rota-Baxter group with limit-weight; $\RR$-group; topological group; Lie group; differential group; integration on groups}

\maketitle

\vspace{-1cm}
\tableofcontents

\setcounter{section}{0}

\allowdisplaybreaks

\section{Introduction}
Using an algebraic interpretation of perturbation in terms of sections, this paper generalizes Rota-Baxter groups from the existing case of weight $\pm 1$ to weights given by (sequences of) pairs of maps, leading to a notion of a Rota-Baxter group with weight zero which has been sought out in recent works. Its relations to Rota-Baxter Lie algebras with weight zero and integration on topological groups, as well as the differential counterparts, are also developed.

\vpb

\subsection{Rota-Baxter algebras and classical Yang-Baxter equation}
The notion of Rota-Baxter algebras have shown its importance in both the associative algebra and Lie algebra contexts.

For the associative algebra, the notion was introduced by G. Baxter in 1960~\cite{Ba}.
Fixing a scalar $\lambda$, a {\bf Rota-Baxter algebra with weight $\lambda$} is a pair $(R,P)$
consisting of an associative algebra $R$ and a linear operator $P:R\to R$ that satisfies the identity
\begin{equation}
	P(x)P(y)=P\big(xP(y)\big)+P\big(P(x)y\big)+\lambda P(xy), \ \quad x,y \in R.
	\label{eq:rb}
\end{equation}
Then $P$ is called a {\bf Rota-Baxter operator (RBO) with  weight $\lambda$}.
The notion of a Rota-Baxter operator with weight zero is an algebraic abstraction of   the Riemann integral operator $I[f](x):=\int_a^xf(t)\,dt$ for which the Rota-Baxter identity is simply the integration-by-parts formula.

Under a linear transformation, a Rota-Baxter operator with weight $1$ becomes a so-called modified Rota-Baxter operator, which have already been employed by the prominent analyst Tricomi in 1951~\cite{Tri} and then by Cotlar in his remarkable 1955 paper~\cite{Co} unifying Hilbert transformations and ergodic theory.
After the pioneering works of Atkinson, Cartier and Rota~\cite{At,Ca,Ro1,Ro2}, Rota-Baxter algebra experienced a remarkable renaissance in the recent decades thanks to its broad applications such as the Connes-Kreimer approach of renormalization in quantum field theory~\cite{CK}. See~\cite{Gub} for further details.

The Lie algebra variations of the Rota-Baxter operator and modified Rota-Baxter operator were discovered by Semenov-Tian-Shansky in~\cite{STS} as the operator forms of the classical Yang-Baxter equation (CYBE)~\cite{BD}. The CYBE arose from the study of
inverse scattering theory in the 1980s and then was recognized as the
``semi-classical limit" of the quantum Yang-Baxter equation following the works of C.~N. Yang~\cite{Ya} and R.~J. Baxter~\cite{BaR}.
CYBE is further related to classical integrable
systems and quantum groups~\cite{CP}.

In the general context of operads, the Rota-Baxter operator plays the role of splitting operations, covering the important notions of pre-Lie algebra, dendriform algebra and post-Lie algebra~\cite{BBGN,Val}.

The differential counterpart of the Rota-Baxter associative algebra is  the differential algebra which was developed initially by Ritt and Kolchin as an algebraic study of differential equations and has expanded into a vast area of mathematical research and applications~\cite{Ko,PS,Ri}.

\subsection{Rota-Baxter groups, braces and quantum Yang-Baxter equation}

Fundamental in applications to integrable systems~\cite{FRS,RS1,RS2}, the Global Factorization Theorem of Semenov-Tian-Shansky for a Lie group was obtained from integrating his Infinitesimal Factorization Theorem for a Lie algebra, making use of the modified Yang-Baxter equation (equivalently, a Rota-Baxter operator with weight $1$).
Thus it is natural to find a Rota-Baxter operator on the Lie group, so that the Global Factorization Theorem can be proved directly on the Lie group level. Such a notion was found in~\cite{GLS}: a {\bf Rota-Baxter operator on a group with weight $1$} is defined to be a map $\frakB:G\to G$ such that
\begin{equation*}
	\frakB(a)\frakB(b)=\frakB\Big(a\frakB(a)b\frakB(a)^{-1} \Big), \quad a, b\in G.
\end{equation*}
This notion has inspired a host of studies involving Rota-Baxter operators on Hopf algebras, pre-Lie groups, post-Lie groups, skew left braces and set-theoretic solutions of the quantum Yang-Baxter equation~\cite{BGST,BG,BN,BNY,CS,CMS,DR,Go,GGLZ,JSZ,LST1,LST2}. Thus a Rota-Baxter operator with weight $1$ on Lie algebras, as an operator form of the classical Yang-Baxter equation, found a quantum analog in Rota-Baxter operators on Lie groups with weight $1$, as an operator form of the set-theoretic solutions of the Yang-Baxter equation.
\vpb

\subsection{Importance of the weight zero case}

Despite the importance of Rota-Baxter operators on groups, the weights of the operators have been restricted to $\pm 1$.
However, the Rota-Baxter operator with  weight zero on a linear structure is still indispensable. For the associative algebra, it is the abstraction of the integration-by-parts formula; while for the Lie algebra, it is the operator form of the CYBE. Furthermore, it gives rise to the notion of pre-Lie algebra, with its own broad applications.
Therefore, finding its group-theoretic counterpart, namely a Rota-Baxter operator with  weight zero on groups, has been the goal of intensive recent research, in which various conventions are made to the weight $\pm 1$ case, such as the abelianness of the groups. Likewise, a post-group is called a pre-group if the underlying Lie group is abelian.
For some of the attempts, see~\cite{BGST,BN,GGLZ,LST1,LST2}.
\vpb

\subsection{The present approach (outline)}
In the present paper and its companion~\cite{GGH2}, we take a different approach to give a new notion of Rota-Baxter groups, including the weight-zero case, that satisfy the expected properties. Both papers built from the more general notion of Rota-Baxter operators with limit-weights. We give one interpretation in the present work, with applications to integration on topological groups. In the companion paper~\cite{GGH2}, we use a limit-abelian condition for relative Rota-Baxter operators with limit-weights to give another interpretation of Rota-Baxter operator with  weight zero,  with applications to pre-groups, braces and Yang-Baxter equations. Similar notions on Hopf algebras and some other algebraic structures can also be defined, generalizing the case with weight $\pm 1$~\cite{Go,LST1}. Some of these directions will be pursued in subsequent works~\cite{GGHZ}.

We next give an outline of this paper. For simplicity, we will focus on Rota-Baxter operators on groups in the outline even though differential operators are also treated in later sections.
We use Figure~\ref{fig:rbalggp} on page~\pageref{fig:rbalggp} to organize the notions introduced, where Rota-Baxter is abbreviated as RB.



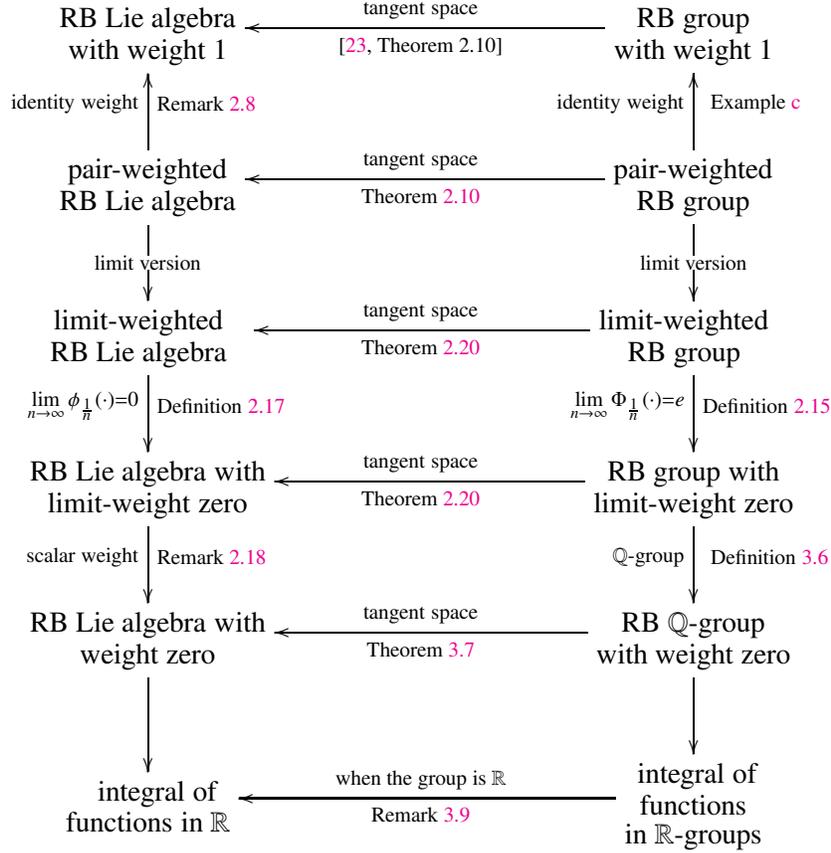
\begin{figure}
	\caption{Rota-Baxter groups and Rota-Baxter Lie algebras}
\label{fig:rbalggp}
\begin{displaymath}
\xymatrix@R=1.0cm{
		\txt{\small RB Lie algebra\\ \small with weight 1} \ar@{<-}[rrrr]^{\txt{\tiny tangent space}}_{\txt{\tiny \cite[Theorem~2.10]{GLS}}}
		&&&&\txt{\small RB group  \\ \small  with weight 1}
		\\
		\txt{\small pair-weighted \\ \small RB Lie algebra}\ar@{<-}[rrrr]^{\txt{\tiny tangent space}}_{\txt{\tiny Theorem~\ref{thm:tgop}}}
		\ar@{->}[u]^{\txt{\tiny  identity weight}}_{\txt{\tiny Remark~\ref{rk:weight lambda}}}
		&&&&\txt{\small pair-weighted \\ \small RB group}\ar@{->}[u]^{\txt{\tiny  identity weight}}_{\txt{ \tiny Example~\ref{exam:pairexam}}}
		\\
		\txt{\small limit-weighted \\ \small RB Lie algebra} \ar@{<-}[rrrr]^{\txt{\tiny tangent space}}_{\txt{\tiny Theorem~\ref{thm:tanglim}}} \ar@{<-}[u]|-{\txt{\tiny limit version}}\hole &&&&\txt{\small limit-weighted \\ \small RB group}\ar@{<-}[u]|-{\txt{\tiny limit version}}\hole
		\\
		\txt{\small RB Lie algebra with\\ \small  limit-weight zero} \ar@{<-}[rrrr]^{\txt{\tiny tangent space}}_{\txt{\tiny Theorem~\ref{thm:tanglim}}} \ar@{<-}[u]^{\tiny \lim\limits_{n\to\infty}\MfrakL_{\frac{1}{n}}(\cdot)=0}_{\txt{\tiny  Definition~\ref{defn:rblielim}}} &&&&\txt{\small RB group with\\ \small  limit-weight zero}\ar@{<-}[u]^{\tiny \lim\limits_{n\to\infty}\mfrakL_{\frac{1}{n}}(\cdot)=e}_{\txt{\tiny Definition~\ref{defn:lhnrb}}}
		\\
		\txt{\small RB Lie algebra with\\ \small  weight zero}\ar@{<-}[rrrr]^{\txt{\tiny tangent space}}_{\txt{\tiny Theorem~\ref{thm:rbg2rbl0}}} \ar@{<-}[u]^{\txt{\tiny scalar weight}}_{\txt{\tiny Remark~\ref{rk:weight
					0}}} &&&& \ar@{<-}[u]^{\txt{\tiny  $\mathbb{Q}$-group}}_{\txt{ \tiny Definition~\ref{defn:rbg0}}} \txt{\small RB $\mathbb{Q}$-group \\ \small with weight zero}
		\\
		\txt{\small \ \, integral of\\ \small functions in $\mathbb{R}$} \ar@{<-}[rrrr]^{\txt{\tiny when the group is $\mathbb{R}$}}_{\txt{\tiny Remark~\ref{rk:Riem}}} \ar@{<-}[u]&&&&\txt{\small\  integral of\\ \small functions\\ \small in $\mathbb{R}$-groups} \ar@{<-}[u]
	}
\end{displaymath}
\end{figure}

\subsubsection{Rota-Baxter groups with pair-weights and limit-weights $($Section~\ref{sec:pairlimit}$)$}
Recall that a section of a surjective map $f:X\to Y$ is a map $g:Y\to X$ such that $fg=\id_Y$. Then $g$ is injective. To mirror the notion, we also call $f$ a cosection of the injective map $g$. We will generalize the Rota-Baxter operator with  weight $1$ on a group by deforming it with a cosection-section pair, and then by a sequence to such pairs to obtain a candidate of the Rota-Baxter operator with  weight zero on groups.

\paragraph{\em Rota-Baxter groups with pair-weights}
Let $G$ be a group, and let  $\mfrakL,\mfrakH:G\to G$ be a pair of maps with $\mfrakL \mfrakH= \id_G$.
Then a map $\mathfrak{B}:G\to G$ is called a {\bf Rota-Baxter operator with pair-weight $(\mfrakL,\mfrakH)$} if
\begin{equation*}
	\mathfrak{B}(a) \mathfrak{B}(b) =  \mathfrak{B} \bigg(\mfrakH \Big( \mfrakL(a) \mathfrak{B}(a) \mfrakL(b) \mathfrak{B}(a)^{-1} \Big) \bigg),\quad  a,b\in G.
	\label{eq:rbgw0}
\end{equation*}
Then the group $G$ is called a {\bf Rota-Baxter group with pair-weight $(\mfrakL,\mfrakH)$} (Definition~\ref{defn:lhrbg}).

The special case when $\mfrakL=\mfrakH=\id_G$ recovers the Rota-Baxter operator with weight $1$ defined in~\cite{GLS}.
In the same way, a Rota-Baxter operator on a Lie algebra with pair-weight $(\MfrakL,\MfrakH)$ is defined and serves as the tangent map of a Rota-Baxter operator with pair-weight on a Lie group (Theorem~\ref{thm:tgop}).

\paragraph{\em Rota-Baxter groups with limit-weights and \limwtzero}

On a topological group $G$, we consider a sequence of pairs of maps $\mfrakL_{\frac{1}{n}}, \mfrakH_{\frac{1}{n}}: G\to G$ with $\mfrakL_{\frac{1}{n}} \mfrakH_{\frac{1}{n}}=\id_G$. such that the limit $\lim\limits_{n\to\infty}\mfrakH_{\frac{1}{n}}\Big(\mfrakL_{\frac{1}{n}}(a)\mfrakL_{\frac{1}{n}}(b)\Big)$ can be obtained in synchronization with approximations of $a$ and $b$ (Definition~\ref{defn:synchronized-limit}). Then a Rota-Baxter operator on $G$ with limit-weight is defined by the operator identity (Definition~\ref{defn:lhnrb})
\begin{equation*}
	\frakB(a) \frakB(b) =  \frakB \bigg(\lim\limits_{n\to \infty}  \mfrakH_{\frac{1}{n}} \Big( \mfrakL_{\frac{1}{n}}(a) \frakB(a) \mfrakL_{\frac{1}{n}}(b) \frakB(a)^{-1} \Big) \bigg).
	\label{eq:rbolimwt0}
\end{equation*}
Under the additional condition of $\lim\limits_{n\to \infty} \mfrakL_{\frac{1}{n}}(a)=e$ for each $a\in G$, where $e$ is the identity of $G$,
we call $\frakB$ a {\bf Rota-Baxter operator with \limwtzero}. This is one of the two candidates for a group-theoretic analog of the Rota-Baxter operator with  weight zero on Lie algebras.

Applying the same process to Lie algebras, we obtain the notions of a Rota-Baxter operator on a Lie algebra with limit-weight and with limit-weight zero, which are again shown to serve as the tangent maps of the Rota-Baxter operators with  limit-weights and limit-weight zero on Lie groups (Theorem~\ref{thm:tanglim}).

\paragraph{\em Realizations as integrals on topological groups}
As noted above, a motivation of the classical notion of Rota-Baxter operators with weight zero is the integration operator as characterized by the integration-by-parts formula. As examples and justifications of our new notion, we develop an integral theory for functions $a:\RR\to G$ which can be regarded as paths in a topological group $G$. We establish that, under a synchronizing condition, the integral operator thus obtained satisfies a multiplicative version of the integration-by-parts formula (Theorem~\ref{thm:HIfml}).

\subsubsection{Rota-Baxter operators on $\RR$-groups $($Section~\ref{sec:qrgroup}$)$}
To make sense of a Rota-Baxter operator on a group with an arbitrary scalar weight, and to provide systematic examples of Rota-Baxter operators with limit-weight zero, in Section~\ref{sec:qrgroup} we introduce a special class of topological groups, called $\RR$-groups for which taking real-valued powers are bijective and consistent with the group structure in an expected manner (Definition~\ref{defn:rgroup}).

By specializing the previous notion of a Rota-Baxter operator with  limit-weight to this case, we define Rota-Baxter operators with real-valued weights whose tangents are the usual Rota-Baxter operators with real-valued weight on Lie algebras (Theorem~\ref{thm:rbg2rbl0}).
We also obtain concrete examples of Rota-Baxter operators with limit-weight zero as integrals of functions with values in $\RR$-groups (Corollary~\ref{co:rgpint}). By further specializing to the matrix groups where the exponential map from its Lie algebra is bijective, the limit-weight zero operator identity can be put in a finite form (Theorem~\ref{thm:rboze}), recovering the notion of a Rota-Baxter operator with  weight zero recently introduced in~\cite{LST2}. Differential operator with various weights are also introduced and equipped with analytic meaning as derivations of $\RR$-group-values functions. Finally a First fundamental Theorem of Calculus is established (Proposition~\ref{pp:fftc}).

\smallskip
\noindent
{\bf Notations.} Throughout this paper, we will work over a fixed field $\bfk$ of characteristic zero. It is the base field for all vector spaces, algebras,  tensor products, as well as linear maps.
Denote by $\mathbb{P}$ the set of positive ones. Let $e$ be the identity of a group.

\section{Pair-weighted and limit-weighted Rota-Baxter groups and differential groups}
\label{sec:pairlimit}
In this section, we first introduce the notion of a Rota-Baxter operator on a group with weight given by a pair of maps. By considering a limit process from an infinite sequence of pairs, we define a Rota-Baxter operator on a group with a limit-weight and then a Rota-Baxter operator on a group with limit-weight zero. We then give some basic properties.

\subsection{Pair-weighted Rota-Baxter groups and Rota-Baxter Lie algebras} \label{ss:rbgw}
We begin with the concept of a Rota-Baxter operator on a group with weight given by a pair of maps on the group. This notion includes as special cases the Rota-Baxter operator on a group with weight $\pm 1$ in~\cite{GLS}, as well as the one in~\cite{BN}.


\begin{defn} Let $G$ be a group and $\mfrakL,\mfrakH:G\rightarrow G$ be maps.
\begin{enumerate}
\item We call $(G,(\mfrakL,\mfrakH))$, or simply $G$,  a {\bf $(\mfrakL,\mfrakH)$-group} if $\mfrakL \mfrakH= \id_G$. \label{it:lhrbg1}

\item For a $(\mfrakL,\mfrakH)$-group $G$, a map $\mathfrak{B}:G\to G$ is called a {\bf Rota-Baxter operator with pair-weight $(\mfrakL,\mfrakH)$} if
\vpb
\begin{equation}
	\mathfrak{B}(a) \mathfrak{B}(b) =  \mathfrak{B} \bigg(\mfrakH \Big( \mfrakL(a) \mathfrak{B}(a) \mfrakL(b) \mathfrak{B}(a)^{-1} \Big) \bigg),\quad  a,b\in G.
	\label{eq:rbgw}
\end{equation}
Then the $(\mfrakL,\mfrakH)$-group $G$
is called a {\bf Rota-Baxter group with pair-weight $(\mfrakL,\mfrakH)$}.
\label{it:lhrbg2}
\item A $(\mfrakL,\mfrakH)$-group $G$ is called a {\bf $(\mfrakL,\mfrakH)$-Lie group} if the group $G$ is a Lie group, and $\mfrakL$ and $\mfrakH$ are smooth maps. \label{it:lhrbg3}
\item If a $(\mfrakL,\mfrakH)$-Lie group $G$ is a Rota-Baxter group with pair-weight $(\mfrakL,\mfrakH)$ for which the Rota-Baxter operator $\frakB$ is a smooth map, then we call $G$ {\bf a Rota-Baxter Lie group with pair-weight $(\mfrakL,\mfrakH)$}.
\end{enumerate}
\label{defn:lhrbg}
\end{defn}
Here we use pair-weight for the Rota-Baxter operator to distinguish it from the scalar weight in the existing literature and the limit-weight to be introduced later.

\begin{exam}
\begin{enumerate}
\item Taking $\mfrakL=\mfrakH: =\id_G$, then Eq.~(\ref{eq:rbgw}) recovers
the concept of a {\bf Rota-Baxter operator with weight $1$} on $G$ introduced in~\cite{GLS}, defined by
\vpb
\begin{equation}
\frakB(a)\frakB(b)=\frakB\Big(a\frakB(a)b\frakB(a)^{-1} \Big), \quad a, b\in G.
\label{eq:rbog}
\vpb
\end{equation}
\item If $\mfrakL(a)=\mfrakH(a):=a^{-1}$ for $a\in G$, then Eq.~(\ref{eq:rbgw}) is reduced to
the concept of a {\bf Rota-Baxter operator with weight $-1$} on $G$ in~\cite{GLS}, defined by
\vpb
$$\frakB(a)\frakB(b)=\frakB\Big(\frakB(a)b\frakB(a)^{-1}a \Big).
\vpb
$$
\item \label{exam:pairexam}
When $\mfrakL$ (and hence $\mfrakH$) is bijective, that is, the pair $(\mfrakL,\mfrakH)$ is bijective as in Definition~\ref{defn:pair} below, and formally denote $\mfrakL(a)=a^\lambda$, we recover the notion of weight $\lambda$ Rota-Baxter operators in~\cite[Proposition~4.1]{BN}. Note that in order for the $\lambda$-power to really make sense, the notion of $\QQ$-groups and $\RR$-groups are needed. See Remark~\ref{exam:wtlambda}.
\end{enumerate}
\end{exam}

For their further study, we impose additional conditions on $(\mfrakL,\mfrakH)$-groups.
\begin{defn}
Let $G$ be a $(\mfrakL,\mfrakH)$-group.
\begin{enumerate}
\item We call the pair $(\mfrakL,\mfrakH)$ {\bf bijective} if $\mfrakL$ is bijective with $\mfrakL^{-1}=\mfrakH$.
\item We call the pair $(\mfrakL,\mfrakH)$ {\bf unital} if $\mfrakL(e)=\mfrakH(e)=e$.
\item
We call the pair $(\mfrakL,\mfrakH)$ a pair of {\bf inverse preserving maps} on $G$ if $\mfrakL(a^{-1})=\mfrakL(a)^{-1}$ and $\mfrakH(a^{-1})=\mfrakH(a)^{-1}$ for each $a \in G$.
\end{enumerate}
\label{defn:pair}
\end{defn}
\vpc

Now we give the relation between Rota-Baxter operators with pair-weights and Rota-Baxter operators with weight $1$.

Let $M(G)$ denote the set of maps from $G$ to itself. It is a unitary semigroup with respect to the composition. For a fixed map $\alpha\in M(G)$, the precomposition with $\alpha$ defines a map
\vpa
\begin{equation*}
\eta_\alpha: M(G)\to M(G), \quad \beta\mapsto \beta\alpha,\ \  \beta\in M(G).
\label{eq:precomp}
\vpa
\end{equation*}
For a $(\mfrakL,\mfrakH)$-group $G$, since $\eta_\mfrakH\,\eta_\mfrakL=\eta_{\mfrakL\mfrakH}=\eta_{\id_G}=\id_{M(G)}$, the map $\eta_\mfrakL$ is injective and the map $\eta_\mfrakH$ is surjective.
Let $\rbo_1$ (resp. $\rbo_{(\mfrakL,\mfrakH)}$) denote the set of Rota-Baxter operators with weight $1$ (resp. pair-weight $(\mfrakL,\mfrakH)$) on $G$.

\begin{prop} \label{pp:pairmap}
	Let $G$ be a $(\mfrakL,\mfrakH)$-group.
\begin{enumerate}
\item \label{it:pairmap4a}
The injection $\eta_\mfrakL$ restricts to a bijection
\vpa
$$\eta_\mfrakL:\rbo_1 \to \rbo_{(\mfrakL,\mfrakH)}\cap \eta_\mfrakL(M(G)),\quad \frakB\mapsto \frakB\mfrakL;
\vpa$$
\item
\label{it:pairmap5}
The surjection $\eta_\mfrakH$ restricts to a map $\eta_\mfrakH: \rbo_{(\mfrakL,\mfrakH)} \to \rbo_1, \frakB\mapsto \frakB\mfrakH$;
\label{it:pairmap4}
\item When $(\mfrakL,\mfrakH)$ is bijective, the map $\eta_\mfrakL$ gives a bijection between $\rbo_{(\mfrakL,\mfrakH)}$ and $\rbo_1$.
\end{enumerate}	
\end{prop}

\begin{proof}
\eqref{it:pairmap4a}
First let $\frakB:G\to G$ be a Rota-Baxter operator with weight $1$. Hence
\vpa
\begin{equation}
\frakB(a)\frakB(b)=\frakB(a \frakB(a)b \frakB(a)^{-1}), \quad a, b\in G.
\label{eq:pairmap1a}
\vpa
\end{equation}
For $a', b'\in G$, taking $a=\mfrakL(a')$ and $b=\mfrakL(b')$ in Eq.~\eqref{eq:pairmap1a} and applying $\mfrakL\mfrakH=\id_G$, we obtain
\vpa
\begin{equation}
\frakB\mfrakL(a')\frakB\mfrakL(b')=\frakB\mfrakL\mfrakH\Big(\mfrakL(a') \frakB\mfrakL(a')\mfrakL(b') \frakB\mfrakL(a')^{-1}\Big), \quad a', b'\in G.
\label{eq:pairmap1b}
\vpa
\end{equation}
This shows that $\frakB\mfrakL$ is a Rota-Baxter operator of pair-weight $(\mfrakL,\mfrakH)$.
Thus we have obtained an injection $\eta_\mfrakL:\rbo_1\to \rbo_{(\mfrakL,\mfrakH)}\cap \eta_{\mfrakL}(M(G))$.

Conversely, suppose that a map $\beta\in M(G)$ is in $\rbo_{(\mfrakL,\mfrakH)}\cap \eta_{\mfrakL}(M(G))$.
Then $\beta=\frakB\mfrakL$ for some $\frakB\in \eta_{\mfrakL}(M(G))$ and is a Rota-Baxter operator of pair-weight $(\mfrakL,\mfrakH)$. Then Eq.~\eqref{eq:pairmap1b} holds for $\beta=\frakB\mfrakL$. Hence Eq.~\eqref{eq:pairmap1a} holds for all $a=\mfrakL(a'), b=\mfrakH(b')\in \mfrakL(G)$. Since $\mfrakL$ is surjective, the identity holds for all $a, b\in G$, showing that $\frakB$ is a Rota-Baxter operator with weight $1$.

\eqref{it:pairmap4}
Let $\frakB:G\to G$ be a Rota-Baxter operator with pair-weight $(\mfrakL,\mfrakH)$. Then
\vpa
\begin{equation}
	\mathfrak{B}(a) \mathfrak{B}(b) =  \mathfrak{B}\mfrakH\Big( \mfrakL(a) \mathfrak{B}(a) \mfrakL(b) \mathfrak{B}(a)^{-1} \Big), \quad a,b\in \mfrakH(G).
	\label{eq:pairmap2b}
\vpa
\end{equation}
Now for $a', b'\in G$, take $a=\mfrakH(a'), b=\mfrakH(b')$ and so $\mfrakL(a)=a', \mfrakL(b)=b'$. Then Eq.~\eqref{eq:pairmap2b} becomes
\vpa
\begin{equation} \mathfrak{B}\mfrakH(a') \mathfrak{B}\mfrakH(b') =  \mathfrak{B}\mfrakH\Big( \mfrakL(a') \mathfrak{B}\mfrakH(a') \mfrakL(b') \mathfrak{B}\mfrakH(a')^{-1} \Big).
	\label{eq:pairmap2}
\vpa
\end{equation}
Hence $\frakB\mfrakH$ is a Rota-Baxter operator of weight $(\mfrakL,\mfrakH)$.
	
\eqref{it:pairmap5}
Under the assumption that $(\mfrakL,\mfrakH)$ is bijective, the two maps in Items~\eqref{it:pairmap4a} and \eqref{it:pairmap4} are the inverses of each other. Hence $\eta_\mfrakL:\rbo_1 \to \rbo_{(\mfrakL,\mfrakH)}$ is a bijection.
\end{proof}

\begin{remark}
\begin{enumerate}
\item For a group $G$, the set of Rota-Baxter operators with weight $1$ is recovered by the set of Rota-Baxter operators with pair-weight $(\mfrakL,\mfrakH)$ when $\mfrakL=\mfrakH=\id_G$. On the other hand, not every Rota-Baxter operator with pair-weight $(\mfrakL,\mfrakH)$ is one with weight $1$.
\item
When the pair $(\mfrakL,\mfrakH)$ is bijective, then by Proposition~\ref{pp:pairmap}, there is a bijection between the set of Rota-Baxter operators with pair-weight $(\mfrakL,\mfrakH)$ and the set of Rota-Baxter operators with weight $1$. In this sense, the study of a Rota-Baxter operator with pair-weight $(\mfrakL,\mfrakH)$ is equivalent to a Rota-Baxter operator with weight $1$. This is the case if $G$ is finite, since then an injection or a surjection is also a bijection. It is also the case considered in~\cite{BN}. We will focus on the case when the pair is not bijective. 	
\end{enumerate}
\end{remark}

We give an example of a Rota-Baxter group with pair-weight that is not a Rota-Baxter group with weight $1$, in which the maps $\mfrakL$ and $\mfrakH$ are group homomorphisms.

\begin{exam} Let $G$ be a group and let $G^\infty$ be the Cartisian product of countably many copies of $G$. Then with the componentwise product, $G^\infty$ is also a group. Note that we have $G^\infty \cong G\times G^\infty$, from which we write $G^\infty=G\times G'$ with $G'=G^\infty$. Define $\mfrakL:G^\infty\to G^\infty$ to be the projection to $G'$:
	$$ \mfrakL:G^\infty =G\times G' \to G' =G^\infty, \quad (a,w)\mapsto w, \quad a\in G, w\in G'=G^\infty.$$
Also define $\mfrakH: G^\infty \to G^\infty$ to be the inclusion of $G^\infty=G'$ into $G^\infty=G\times G'$:
	$$\mfrakH: G^\infty \to G' \subseteq G\times G', w\mapsto (e,w), \quad w\in G^\infty.$$
Then $\mfrakL\mfrakH=\id_G$ and so $G$ is a $(\mfrakL,\mfrakH)$-group.
	
Define
$\frakB: G^\infty \to G^\infty,\,  (a,w)\mapsto w^{-1}.$
	Then $\frakB\mfrakH(w)=\frakB(e,w)=w^{-1}$ and hence $\frakB\mfrakH$ is a Rota-Baxter operator with weight $1$ on $G^{\infty}$~\cite{GLS}. Proposition~\ref{pp:pairmap}~\eqref{it:pairmap4} then suggests that $\frakB$ is a Rota-Baxter operator with pair-weight $(\mfrakL,\mfrakH)$.
	
To verify this, for $(a,w), (b,u)\in G^\infty$ with $a, b\in G$ and $w,u\in G'$, we have
	\begin{equation}
		\frakB(a,w)\frakB(b,u)=w^{-1}u^{-1},
		\label{eq:cexamlhs}
	\end{equation}
	and
	$$
	\frakB\Big(\mfrakH(\mfrakL(a,w)\frakB(a,w)\mfrakL(b,u)\frakB(a,w)^{-1})\Big)
	= \frakB(\mfrakH( w w^{-1} u w))
	=\frakB(\mfrakH(uw))
	=\frakB(e,uw)=(uw)^{-1}.$$
	This agrees with Eq.~\eqref{eq:cexamlhs} and verifies Eq.~\eqref{eq:rbgw}, showing that $\frakB$ is indeed a Rota-Baxter operator with  pair-weight $(\mfrakL,\mfrakH)$.
	
However, in order for $\frakB$ to be a Rota-Baxter operator with weight $1$, we need to check Eq.~\eqref{eq:rbog}.
Let $(a,w)=(a,a',w'), (b,u)=(b,b',u')$ with $a,a',b,b'\in G$ and $w',u'\in G^\infty$. Then the right hand side of Eq.~\eqref{eq:rbog} for the Rota-Baxter relation with weight $1$ is
\begin{align*}
&\ \frakB\big((a,w)\frakB(a,w)(b,u)\frakB(a,w)^{-1}\big)
=\frakB\big((a,w)((a')^{-1},(w')^{-1})(b,u)(a',w')\big)\\
=&\ \frakB(a(a')^{-1}ba',w(w')^{-1}uw')
=(w(w')^{-1}uw')^{-1}=(w')^{-1}u^{-1}w'w^{-1}.
\end{align*}
This does not agree with Eq.~\eqref{eq:cexamlhs} when $G$ is not abelian. Hence $\frakB$ is not a Rota-Baxter operator with weight $1$.
\end{exam}

We also define Rota-Baxter Lie algebras with weights given by a pair of linear maps.
\begin{defn}
Let $\frakg$ be a Lie algebra and $\MfrakL,\MfrakH, B:\frakg \rightarrow \frakg$ be linear maps.
\begin{enumerate}
\item We call $(\frakg,(\MfrakL,\MfrakH))$ a {\bf $(\MfrakL,\MfrakH)$-Lie algebra} if $\MfrakL \MfrakH= \id_\frakg$.
\item A linear operator $B$ on a $(\phi,\psi)$-Lie algebra $\frakg$ is called a {\bf Rota-Baxter operator with pair-weight $(\MfrakL,\MfrakH)$}
if
\begin{equation}
[B(u),B(v)] =  B\Big(\MfrakH\big([B(u),\MfrakL(v)]+[\MfrakL(u),B(v)]+[\MfrakL(u),\MfrakL(v)]\big) \Big), \quad  u,v\in \frakg.
\label{eq:algpair}
\end{equation}
Then the triple $(\frakg,(\MfrakL,\MfrakH),B)$ is called a {\bf Rota-Baxter Lie algebra with pair-weight $(\MfrakL,\MfrakH)$}.  \end{enumerate}
\label{defn:wrbliealg}
\end{defn}

\begin{remark}
 Let $\lambda\in {\bfk} \backslash \{0\}$, $\MfrakL :=\lambda\, \id_\frakg$ and $\MfrakH:=\frac{1}{\lambda} \id_\frakg$. Then a Rota-Baxter Lie algebra with weight $(\MfrakL,\MfrakH)$ is reduced to the usual Rota-Baxter Lie algebra with weight $\lambda$.
\label{rk:weight lambda}
\end{remark}

We next establish a relation between Rota-Baxter groups with pair-weights and Rota-Baxter Lie algebras with pair-weights, beginning with a lemma.

\begin{lemma}
Let $G$ be a $(\mfrakL,\mfrakH)$-Lie group with $(\mfrakL,\mfrakH)$ a unital pair. Let $\frakg = T_e G$ be the Lie algebra of $G$ and $\mfrakL_{*e},\mfrakH_{*e}:\frakg\to \frakg$ the tangent maps of $\mfrakL$ and $\mfrakH$ at the identity $e$ respectively. Then $\frakg$ is a $(\mfrakL_{*e},\mfrakH_{*e})$-Lie algebra.
\label{lemma:idl}
\end{lemma}

\begin{proof}
For each $u\in \frakg$,  we have
\begin{align*}
u&=\frac{\mathrm{d}}{\mathrm{d}t }\bigg|_{t=0} e^{tu}
=\frac{\mathrm{d}}{\mathrm{d}t }\bigg|_{t=0} \mfrakL\mfrakH\Big(e^{tu}\Big)
=\mfrakL_{*e}\bigg(\frac{\mathrm{d}}{\mathrm{d}t }\bigg|_{t=0} \mfrakH\Big(e^{tu}\Big)\bigg)
=\mfrakL_{*e}\mfrakH_{*e}\Big(\frac{\mathrm{d}}{\mathrm{d}t }\bigg|_{t=0} e^{tu}\Big)
=\mfrakL_{*e}\mfrakH_{*e}(u),
\end{align*}
and so $ \mfrakL_{*e}\mfrakH_{*e}=\id_\frakg$.
\end{proof}

\begin{theorem}
Let $(G, \frakB)$ be a Rota-Baxter Lie group with pair-weight $(\mfrakL,\mfrakH)$ where $(\mfrakL,\mfrakH)$ is unital, and let $\frakg = T_e G$ be the Lie algebra of $G$. Let the linear operators
$B :=\frakB_{*e},\, \MfrakL :=\mfrakL_{*e},\, \MfrakH :=\mfrakH_{*e}$
on $\frakg$ be the tangent maps of $\frakB$, $\mfrakL$ and $\mfrakH$ at the identity $e$, respectively. Then $(\frakg, B)$ is a Rota-Baxter Lie algebra with pair-weight $(\MfrakL,\MfrakH)$.
\label{thm:tgop}
\end{theorem}

\begin{proof}
Since $\frakg$ is a $(\MfrakL,\MfrakH)$-Lie algebra by Lemma~\ref{lemma:idl}, we just need to verify Eq.~\eqref{eq:algpair} as follows.
{\small
\begin{align*}
\left [ B(u), B(v) \right ]
=&\  \frac{\mathrm{d}^2}{\mathrm{d}t \mathrm{d}s}\bigg|_{t,s=0} e^{tB(u)}e^{sB(v)}e^{-tB(u)}\\
=&\  \frac{\mathrm{d}^2}{\mathrm{d}t \mathrm{d}s}\bigg|_{t,s=0} \frakB(e^{tu}) \frakB(e^{sv}) \frakB(e^{-tu})\\
=&\  \frac{\mathrm{d}^2}{\mathrm{d}t \mathrm{d}s}\bigg|_{t,s=0} \frakB\bigg( \mfrakH\Big( \mfrakL(e^{tu}) \frakB(e^{tu}) \mfrakL(e^{sv}) \frakB(e^{tu})^{-1} \Big) \bigg)\frakB(e^{-tu})\\
=&\  \frac{\mathrm{d}^2}{\mathrm{d}t \mathrm{d}s}\bigg|_{t,s=0} \frakB\mfrakH \Big( \mfrakL(e^{tu}) \frakB(e^{tu}) \mfrakL(e^{sv}) \frakB(e^{sv}) \mfrakL(e^{-tu}) \frakB(e^{sv})^{-1} \frakB(e^{tu})^{-1}\Big)\\
=&\  \frakB_{*e}\mfrakH_{*e}\Bigg(\frac{\mathrm{d}^2}{\mathrm{d}t \mathrm{d}s}\bigg|_{t,s=0}\frakB(e^{tu})\mfrakL(e^{sv})\frakB(e^{tu})^{-1}+\frac{\mathrm{d}^2}{\mathrm{d}t \mathrm{d}s}\bigg|_{t,s=0}\frakB(e^{sv})\mfrakL(e^{-tu})\frakB(e^{sv})^{-1}\\ & \quad \quad \quad+\frac{\mathrm{d}^2}{\mathrm{d}t \mathrm{d}s}\bigg|_{t,s=0}\mfrakL(e^{tu})\mfrakL(e^{sv})\mfrakL(e^{-tu}) \Bigg)\\
=&\    B\bigg(\MfrakH\Big([B(u),\MfrakL(v)]+[\MfrakL(u),B(v)]+[\MfrakL(u),\MfrakL(v)]\Big) \bigg). \qedhere
\end{align*}
}
\end{proof}

We give more examples of Rota-Baxter groups with pair-weights.

\begin{prop}
Let $(G, (\mfrakL,\mfrakH))$ be a $(\mfrakL,\mfrakH)$-group, $G_{+}$ and $G_{-}$ be two subgroups with $G$ such that $G = G_{+}G_{-}$ and $G_{+}\cap G_{-} = \{ e\}$. For $a\in G$, denote $\mfrakL(a)  =\mfrakL(a)_{+}\mfrakL(a)_{-}$ with $\mfrakL(a)_{+}\in G_{+}$ and $\mfrakL(a)_{-}\in G_{-}$. For the operator
$$\frakB: G \rightarrow G, \quad a \mapsto (\mfrakL(a)_{-})^{-1},
$$
the pair $(G,\frakB)$ is a Rota-Baxter group with pair-weight $(\mfrakL, \mfrakH)$.
\label{Prop:dkrb1}
\end{prop}

\begin{proof}
Let $\mfrakL(a)=\mfrakL(a)_{+}\mfrakL(a)_{-}$ and $\mfrakL(b)=\mfrakL(b)_{+}\mfrakL(b)_{-}$ be two elements in $G$ with $\mfrakL(a)_{+},\mfrakL(b)_{+} \in G_{+}$ and $\mfrakL(a)_{-},\mfrakL(b)_{-} \in G_{-}$. Then
\begin{align*}
\frakB\bigg(\mfrakH\Big(\mfrakL(a)\frakB(a)\mfrakL(b)\frakB(a)^{-1}\Big)\bigg) = &\
\frakB\bigg(\mfrakH\Big(\mfrakL(a)_{+} \mfrakL(a)_{-} (\mfrakL(a)_{-})^{-1}\mfrakL(b)_{+}
\mfrakL(b)_{-}\mfrakL(a)_{-}\Big)\bigg)\\
= &\ \frakB\bigg(\mfrakH\Big(\mfrakL(a)_{+}\mfrakL(b)_{+}\mfrakL(b)_{-}\mfrakL(a)_{-}\Big)\bigg)\\
=&\ (\mfrakL(b)_{-}\mfrakL(a)_{-})^{-1} \hspace{1cm} \text{(by the definition of $\frakB$ and $\mfrakL\mfrakH = \id_G$)}\\
=&\   (\mfrakL(a)_{-})^{-1} (\mfrakL(b)_{-})^{-1} \\
=&\  \frakB(a)\frakB(b). \qedhere
\end{align*}
\end{proof}

As in the case of Rota-Baxter algebras~\cite[Proposition 2.5]{GLS}, a pair-weighted Rota-Baxter operator on a group $G$
induces another Rota-Baxter operator on $G$ of the same pair-weight.

\begin{prop}
Let $(G, \frakB)$ be a Rota-Baxter group with pair-weight $(\mfrakL,\mfrakH)$, where
$(\mfrakL,\mfrakH)$ is a pair of inverse preserving maps. Define
\begin{equation}
\widetilde{\frakB}:G\rightarrow G, \quad a\mapsto \mfrakL(a^{-1})\mathfrak{B}(a^{-1}).
	\label{eq:trb0}
\end{equation}
Then $(G, \widetilde{\frakB})$ is a Rota-Baxter group with pair-weight $(\mfrakL,\mfrakH)$.
\label{pp:trb}
\end{prop}

\begin{proof}
The conclusion follows from
{\small
\begin{align*}	
&\ \widetilde{\frakB}(a)\widetilde{\frakB}(b)\\
=&\ \mfrakL(a^{-1})\frakB(a^{-1})\mfrakL(b^{-1})\frakB(b^{-1})= \mfrakL(a^{-1})\frakB(a^{-1})
\mfrakL(b^{-1})\frakB(a^{-1})^{-1}
\frakB(a^{-1})\frakB(b^{-1})\\
=&\ \mfrakL(a^{-1})\frakB(a^{-1})\mfrakL(b^{-1})
\frakB(a^{-1})^{-1}\frakB\Big(\mfrakH(\mfrakL(a^{-1})
\frakB(a^{-1})\mfrakL(b^{-1})\frakB(a^{-1})^{-1})\Big)\\
=&\ \mfrakL\bigg(\mfrakH\Big(\mfrakL(a^{-1})\frakB(a^{-1})\mfrakL(b^{-1})
\frakB(a^{-1})^{-1}\Big)\bigg)\frakB\bigg(\mfrakH(\mfrakL(a^{-1})
\frakB(a^{-1})\mfrakL(b^{-1})\frakB(a^{-1})^{-1})\bigg) \quad \text{(by $\mfrakL \mfrakH=\id_G$)}\\
=&\ \widetilde{\frakB}\bigg(\mfrakH\Big(\mfrakL(a^{-1})
\frakB(a^{-1})\mfrakL(b^{-1})\frakB(a^{-1})^{-1}
\Big)^{-1}\bigg) \hspace{1cm} \text{(by Eq.~(\ref{eq:trb0}))}\\
=&\ \widetilde{\frakB}\bigg(\mfrakH\Big(\mfrakL(a)^{-1}
\frakB(a^{-1})\mfrakL(b)^{-1}\frakB(a^{-1})^{-1}
\Big)^{-1}\bigg)  \hspace{1cm}\text{(by $\mfrakL$ being inverse preserving)}\\
=&\ \widetilde{\frakB}\bigg(\mfrakH\Big(
\frakB(a^{-1})\mfrakL(b)\frakB(a^{-1}
)^{-1}\mfrakL(a) \Big)\bigg) \hspace{1cm}\text{(by $\mfrakH$ being inverse preserving)}\\
=&\ \widetilde{\frakB}\bigg(\mfrakH\Big( \mfrakL(a)\mfrakL(a)^{-1}
\frakB(a^{-1}) \mfrakL(b) \frakB(a^{-1}
)^{-1}\mfrakL(a) \Big)\bigg)  \\
=&\ \widetilde{\frakB}\bigg(\mfrakH\Big(\mfrakL(a)\mfrakL(a^{-1})
\frakB(a^{-1})\mfrakL(b)\big(\mfrakL(a^{-1}) \frakB(a^{-1})\big)^{-1} \Big)\bigg)\\
=&\ \widetilde{\frakB}\bigg(\mfrakH\Big(\mfrakL(a)\widetilde{\frakB}
(a)\mfrakL(b)\widetilde{B}(a)^{-1}\Big)\bigg). \qedhere
\end{align*}
}
\end{proof}

\subsection{Limit-weighted Rota-Baxter groups and Rota-Baxter Lie algebras} \label{ss:rbgwl}
Our next notions are based on a condition where an iterated limit can be obtained by taking the limit in the diagonal direction. Since the limit in a topological space might not be unique, we will emphasize the uniqueness condition of a limit.

\begin{defn}
\label{defn:synchronized-limit}
Let $T$ be a topological space and $f_n: T\times T\rightarrow T, n\in \PP,$ be a sequence of maps such that $\lim\limits_{n\to \infty} f_n(a,b)$ uniquely exists for each $(a,b)\in T\times T$. We call the resulting limit function $\lim\limits_{n\to\infty}f_n:T\times T\to T$ {\bf \complim}
if, for any sequences $ a_{n},b_{n}, n\in \PP$ in $T$ with $\lim\limits_{n\to \infty}a_{n}= a$ and $\lim\limits_{n\to\infty} b_{n}=b$, we have
\begin{equation}
\label{eq:jx}
\lim_{n\to\infty}f_n(a,b)=\lim_{n\to\infty}f_n(\lim_{m\to \infty} a_m,\lim_{k\to \infty}b_k)=\lim_{n\to\infty}f_n(a_n,b_n).
\end{equation}
\end{defn}

\begin{exam}
Let $V$ be a linear metric space and $f_{n}:V\times V\rightarrow V$ be maps  uniformly converging to a continuous function $f$. Then $f$ is \complim. Indeed, for all $a_n,b_n\in V$ with $\lim\limits_{n\to\infty}a_n=a$ and $\lim\limits_{n\to\infty}b_n=b$, and for each given $ \epsilon > 0$, there exist $N_1,N_2\in \PP$, such that for any $ n \geq N_1$, there are 
    \begin{align*}
        \max_{(a',b')\in V\times V}|f(a',b')-f_n(a',b')|\leq \frac{\epsilon}{2},
    \end{align*}
    and, for any $ n \geq N_2$, there are 
    \begin{align*}
        |f(a,b)-f(a_n,b_n)|\leq \frac{\epsilon}{2}.
    \end{align*}
    So, for each given $ \epsilon > 0$ and $ n \geq \max\{N_1,N_2\}$, there are
    \begin{align*}
        |f(a,b)-f_n(a_n,b_n)|\leq |f(a,b)-f(a_n,b_n)| + |f(a_n,b_n)-f_n(a_n,b_n)|\leq \epsilon. 
    \end{align*}
    Thus we have $\lim\limits_{n\to\infty}f_n(a,b)=\lim\limits_{n\to\infty}f_n(a_n,b_n)$.
\end{exam}

We now present a limit version of the Rota-Baxter operators with pair-weights.

\begin{defn} \label{defn:lhnrb}
Let $(G,\cdot)$ be a topological group, and $\mfrakL_{\frac{1}{n}},\mfrakH_{\frac{1}{n}},\frakB:G\rightarrow G, n\in \PP$, be maps.
\begin{enumerate}
\item We call $\big(G,(\mfrakL_{\frac{1}{n}},\mfrakH_{\frac{1}{n}})\big)$  a  {\bf $\lim\limits_{n\to \infty}(\mfrakL_{\frac{1}{n}},\mfrakH_{\frac{1}{n}})$-group}
if $G$ is a $(\mfrakL_{\frac{1}{n}}, \mfrakH_{\frac{1}{n}})$-group for each $n\in \PP$, the limit $\lim\limits_{n\to\infty}\mfrakH_{\frac{1}{n}}\Big(\mfrakL_{\frac{1}{n}}(a)
\mfrakL_{\frac{1}{n}}(b)\Big)$ uniquely exists for each $(a,b)\in G\times G$, and the map $$\lim\limits_{n\to\infty}\mfrakH_{\frac{1}{n}}\Big(\mfrakL_{\frac{1}{n}}\cdot
\mfrakL_{\frac{1}{n}}\Big): G\times G \rightarrow G$$ is \complim.
\label{it:lhnrb1}

\item For a $\lim\limits_{n\to \infty}(\mfrakL_{\frac{1}{n}},\mfrakH_{\frac{1}{n}})$-group $G$, a map $\mathfrak{B}:G\to G$ is called a {\bf Rota-Baxter operator with limit-weight} $\lim\limits_{n\to \infty}(\mfrakL_{\frac{1}{n}},\mfrakH_{\frac{1}{n}})$ if for each $(a,b)\in G\times G$, the limit
$$\lim\limits_{n\to \infty}  \mfrakH_{\frac{1}{n}} \Big( \mathfrak{B}(a) \mfrakL_{\frac{1}{n}}(b) \mathfrak{B}(a)^{-1} \Big) $$
uniquely exists and there is the identity
\begin{equation}
\frakB(a) \frakB(b) =  \frakB \bigg(\lim\limits_{n\to \infty}  \mfrakH_{\frac{1}{n}} \Big( \mfrakL_{\frac{1}{n}}(a) \frakB(a) \mfrakL_{\frac{1}{n}}(b) \frakB(a)^{-1} \Big) \bigg).
\label{eq:limitwt}
\end{equation}
Then the group $G$ is called a {\bf Rota-Baxter group with limit-weight $\lim\limits_{n\to \infty}(\mfrakL_{\frac{1}{n}},\mfrakH_{\frac{1}{n}})$}.
\item For a Rota-Baxter group $(G,\frakB)$ with limit-weight $\lim\limits_{n\to \infty}(\mfrakL_{\frac{1}{n}},\mfrakH_{\frac{1}{n}})$, if in addition, the limit $\lim\limits_{n\to \infty} \mfrakL_{\frac{1}{n}}(a)$ uniquely exists and equals to $e$ for each $a\in G$,
then we call $\frakB$ (resp. $(G,\frakB)$) a {\bf Rota-Baxter operator} (resp. {\bf Rota-Baxter group}) {\bf with limit-weight zero}.
\item A $\lim\limits_{n\to\infty}(\mfrakL_{\frac{1}{n}},\mfrakH_{\frac{1}{n}})$-group $G$ is called a {\bf $\lim\limits_{n\to\infty}(\mfrakL_{\frac{1}{n}},\mfrakH_{\frac{1}{n}})$-Lie group} if  the group $G$ is a Lie group and $\mfrakL_{\frac{1}{n}}$ and $\mfrakH_{\frac{1}{n}}$ are smooth maps.
\item \label{it:lhnrb2}
A Rota-Baxter group with  limit-weight    $\lim\limits_{n\to\infty}(\mfrakL_{\frac{1}{n}},\mfrakH_{\frac{1}{n}})$ (resp. limit-weight zero) in which $G$ is a Lie group and $\mfrakL_{\frac{1}{n}},\mfrakH_{\frac{1}{n}},\frakB$ are smooth maps, is called a {\bf Rota-Baxter Lie group with limit-weight $\lim\limits_{n\to\infty}(\mfrakL_{\frac{1}{n}},\mfrakH_{\frac{1}{n}})$} (resp. {\bf limit-weight zero}).
\end{enumerate}
\end{defn}

\begin{remark}
\begin{enumerate}
\item The condition $\lim\limits_{n\to \infty}  \mfrakH_{\frac{1}{n}} \Big( \mathfrak{B}(a) \mfrakL_{\frac{1}{n}}(b) \mathfrak{B}(a)^{-1} \Big)$  implies that the limit 
\begin{align*}
\lim\limits_{n\to \infty}  \mfrakH_{\frac{1}{n}} \Big( \mfrakL_{\frac{1}{n}}(a)\mathfrak{B}(a) \mfrakL_{\frac{1}{n}}(b) \mathfrak{B}(a)^{-1} \Big)= &\ \lim\limits_{n\to \infty}  \mfrakH_{\frac{1}{n}} \Bigg( \mfrakL_{\frac{1}{n}}(a)\mfrakL_{\frac{1}{n}}\bigg(\mfrakH_{\frac{1}{n}} \Big( \mathfrak{B}(a) \mfrakL_{\frac{1}{n}}(b) \mathfrak{B}(a)^{-1} \Big)\bigg)  \Bigg)\\
=&\ \lim\limits_{n\to \infty}  \mfrakH_{\frac{1}{n}} \Bigg( \mfrakL_{\frac{1}{n}}(a)\mfrakL_{\frac{1}{n}}\bigg(\lim_{m\to\infty}\mfrakH_{\frac{1}{m}} \Big( \mathfrak{B}(a) \mfrakL_{\frac{1}{m}}(b) \mathfrak{B}(a)^{-1} \Big)\bigg)  \Bigg)
\end{align*} uniquely exists. 
\item In the definition of a Rota-Baxter group with limit-weight $\lim\limits_{n\to \infty}(\mfrakL_{\frac{1}{n}},\mfrakH_{\frac{1}{n}})$, we do not require that $\frakB$ is a Rota-Baxter operator with pair-weight $(\mfrakL_{\frac{1}{n}},\mfrakH_{\frac{1}{n}})$ for each $n\in \PP$, or that either of the limits $\lim\limits_{n\to \infty}\mfrakL_{\frac{1}{n}}(a)$ and $\lim\limits_{n\to \infty}\mfrakH_{\frac{1}{n}}(a)$ uniquely exists for each $a\in G$. The limit $\lim\limits_{n\to \infty}\mfrakL_{\frac{1}{n}}(a)$ is required to uniquely exist (and equal to $e$) only for a Rota-Baxter group with {\it limit-weight zero}.

\item For a topological group $G$, a Rota-Baxter operator with pair-weight $(\mfrakL,\mfrakH)$ is a special case of a Rota-Baxter group with limit-weight $\lim\limits_{n\to \infty}(\mfrakL_{\frac{1}{n}},\mfrakH_{\frac{1}{n}})$ by taking  $\mfrakL_{\frac{1}{n}} :=\mfrakL$ and $\mfrakH_{\frac{1}{n}} :=\mfrakH$ for all $n\in \mathbb{P}$.
\end{enumerate}
\end{remark}

We will give an example of a Rota-Baxter operator with limit-weight zero on a group in Proposition~\ref{prop:rb33} after giving the notion of a $\mathbb{Q}$-group. For the moment, we present the Lie algebra version of the above notions.

\begin{defn}
Let $\frakg$ be a topological Lie algebra and $\MfrakL_{\frac{1
}{n}},\MfrakH_{\frac{1}{n}}, B:\frakg \rightarrow \frakg, n\in \PP,$ be linear maps.
\begin{enumerate}
\item We call $\big(\frakg,\big(\MfrakL_{\frac{1}{n}},\MfrakH_{\frac{1}{n}}\big)\big)$  a {\bf $\lim\limits_{n\to\infty}(\MfrakL_{\frac{1}{n}},\MfrakH_{\frac{1}{n}})$-Lie algebra} if $\frakg$ is a $(\MfrakL_{\frac{1}{n}}, \MfrakH_{\frac{1}{n}})$-Lie algebra for all $n\in \mathbb{P}$, the limit $\lim\limits_{n\to\infty}\MfrakH_{\frac{1}{n}}\Big([\MfrakL_{\frac{1}{n}}(u),\MfrakL_{\frac{1}{n}}(v)]\Big)$ for $(u,v)\in \frakg\times \frakg$ uniquely exists, and the map $\lim\limits_{n\to\infty}\MfrakH_{\frac{1}{n}}\Big([\MfrakL_{\frac{1}{n}},\MfrakL_{\frac{1}{n}}]\Big):\frakg \times \frakg \rightarrow \frakg$  is \complim.
\item For a $\lim\limits_{n\to\infty}(\MfrakL_{\frac{1}{n}},\MfrakH_{\frac{1}{n}})$-Lie algebra $\frakg$, a linear operator $B$ on $\frakg$
is called a {\bf Rota-Baxter operator with limit-weight $\lim\limits_{n\to\infty}(\MfrakL_{\frac{1}{n}},\MfrakH_{\frac{1}{n}})$} if for any $u,v\in \frakg$, the limit
$$\lim\limits_{n\to \infty}\MfrakH_{\frac{1}{n}}\Big([B(u),\MfrakL_{\frac{1}{n}}(v)]\Big)$$
uniquely exists and the identity
\begin{equation}
[B(u),B(v)] =  B\Big(\lim_{n\to \infty}\MfrakH_{\frac{1}{n}}\Big([B(u),\MfrakL_{\frac{1}{n}}(v)]+[\MfrakL_{\frac{1}{n}}(u),B(v)]+[\MfrakL_{\frac{1}{n}}(u),\MfrakL_{\frac{1}{n}}(v)]\Big) \Big)
\label{eq:rblim}
\end{equation}
holds. In this case, we call $(\frakg,B)$ a {\bf Rota-Baxter Lie algebra with limit-weight $\lim\limits_{n\to\infty}(\MfrakL_{\frac{1}{n}},\MfrakH_{\frac{1}{n}})$}.
\item A Rota-Baxter Lie algebra $(\frakg,B)$ with limit-weight $\lim\limits_{n\to\infty}(\MfrakL_{\frac{1}{n}},\MfrakH_{\frac{1}{n}})$ that satisfies  $\lim\limits_{n\to\infty}\MfrakL_{\frac{1}{n}}(u)=0$  for all $u\in \frakg$ is called a {\bf Rota-Baxter Lie algebra with \limwtzero}.
\delete{
\item
A Rota-Baxter Lie algebra $(\frakg,B)$ with \limwtzero that satisfies $\lim\limits_{n\to\infty}[\MfrakL_{\frac{1}{n}}(u),\MfrakL_{\frac{1}{n}}(v)]=0$ for all $u,v\in \frakg$ is called a {\bf limit-abelian Rota-Baxter Lie algebra with limit-weight zero}.
}
\end{enumerate}
\label{defn:rblielim}
\end{defn}

\begin{remark}
As can be checked directly in Eq.~\eqref{eq:rblim}, by taking
$\MfrakL_{\frac{1}{n}}:=\frac{1}{n} \id_\frakg$ (resp. $\MfrakL_{\frac{1}{n}}:=\lambda \id_\frakg$) and $\MfrakH_{\frac{1}{n}}:=n\, \id_\frakg$ (resp. $\MfrakH_{\frac{1}{n}}:=\frac{1}{\lambda} \id_\frakg$), a Rota-Baxter Lie algebra with limit-weight $\lim\limits_{n\to\infty}(\MfrakL_{\frac{1}{n}},\MfrakH_{\frac{1}{n}})$
reduces to the usual Rota-Baxter Lie algebra with weight $0$ (resp. weight $\lambda$).
\label{rk:weight 0}
\end{remark}

The following is the limit version of Theorem~\ref{thm:tgop}. Before that, we need a lemma. Let $G$ be a Lie group and $\frakg$ its Lie algebra. We equip a topology on $\frakg$ by taking the open sets of $\frakg$ to be the inverse images of the open sets of $G$ under the map $\exp:\frakg\to G$, called the {\bf topology on $\frakg$ induced from its Lie group $G$}.

\begin{lemma}
Let $G$ be a $\lim\limits_{n\to\infty}(\mfrakL_{\frac{1}{n}},\mfrakH_{\frac{1}{n}})$-Lie group with $(\mfrakL_{\frac{1}{n}},\mfrakH_{\frac{1}{n}})$ a unital pair for each $n\in \mathbb{P}$. Let $\frakg = T_e G$ be the Lie algebra of $G$, and let $
 \MfrakL_{\frac{1}{n}}:=(\mfrakL_{\frac{1}{n}})_{*e}$ and $\MfrakH_{\frac{1}{n}}:=(\mfrakH_{\frac{1}{n}})_{*e}$
be the tangent maps of $\mfrakL_{\frac{1}{n}}$ and $\mfrakH_{\frac{1}{n}}$ at the identity $e$, respectively. Then with the induced topology from $G$, $\frakg$ is a $\lim\limits_{n\to\infty}(\MfrakL_{\frac{1}{n}},\MfrakH_{\frac{1}{n}})$-Lie algebra.
\label{lemma:groupyr}
\end{lemma}

\begin{proof}
By Lemma~\ref{lemma:idl}, $\frakg$ is a $(\MfrakL_{\frac{1}{n}},\MfrakH_{\frac{1}{n}})$-Lie algebra for $n\in\PP$. Further,
{\small
\begin{align*}
&\ \frac{\mathrm{d}^2}{\mathrm{d}t \mathrm{d}s}\bigg|_{t,s=0}\lim_{n\to\infty}\mfrakH_{\frac{1}{n}}
\Big(\mfrakL_{\frac{1}{n}}(e^tu)\mfrakL_{\frac{1}{n}}(e^sv)\Big)-\frac{\mathrm{d}^2}{\mathrm{d}t \mathrm{d}s}\bigg|_{t,s=0}\lim_{n\to\infty}\mfrakH_{\frac{1}{n}}
\Big(\mfrakL_{\frac{1}{n}}(e^sv)\mfrakL_{\frac{1}{n}}(e^tu)\Big)\\
=&\  \lim_{n\to\infty}(\mfrakH_{\frac{1}{n}})_{*e}\Big(\frac{\mathrm{d}^2}{\mathrm{d}t \mathrm{d}s}\bigg|_{t,s=0}e^{\MfrakL(tu)}e^{\mfrakL(sv)}-\frac{\mathrm{d}^2}{\mathrm{d}t \mathrm{d}s}\bigg|_{t,s=0}e^{\MfrakL(sv)}e^{\mfrakL(tu)}\Big)\\
= &\ \lim_{n\to\infty}(\mfrakH_{\frac{1}{n}})_{*e}\Big(\frac{\mathrm{d}^2}{\mathrm{d}t \mathrm{d}s}\bigg|_{t,s=0}e^{\MfrakL(tu)+\mfrakL(sv)+\frac{1}{2}[\mfrakL(tu),\mfrakL(sv)]+\cdots}-\frac{\mathrm{d}^2}{\mathrm{d}t \mathrm{d}s}\bigg|_{t,s=0}e^{\MfrakL(tu)+\mfrakL(sv)+\frac{1}{2}[\mfrakL(sv),\mfrakL(tu)]+\cdots}\Big)\\
&\ \quad \text{(by the Baker-Campbell-Hausdorff formula)}\\
= &\ \lim_{n\to\infty}\MfrakH_{\frac{1}{n}}
\Big([\MfrakL_{\frac{1}{n}}(u),\MfrakL_{\frac{1}{n}}(v)]\Big).
\end{align*}
}
Hence the limit uniquely exists.
It remains to check that the limit is \complim. Indeed, for $u_n,v_n \in \frakg$ with $\lim\limits_{n\to\infty}u_n=u$ and $\lim\limits_{n\to\infty}v_n=v$, we have
{\small
\begin{align*}
&\ \lim_{n\to\infty}\MfrakH_{\frac{1}{n}}
\Big([\MfrakL_{\frac{1}{n}}(u),\MfrakL_{\frac{1}{n}}(v)]\Big)\\
=&\  \lim_{n\to\infty}\Big(\mfrakH_{\frac{1}{n}}\Big)_{*e}\Big(\frac{\mathrm{d}^2}{\mathrm{d}t \mathrm{d}s}\bigg|_{t,s=0}
\mfrakL_{\frac{1}{n}}(e^{tu})\mfrakL_{\frac{1}{n}}(e^{sv})-\frac{\mathrm{d}^2}{\mathrm{d}t \mathrm{d}s}\bigg|_{t,s=0}
\mfrakL_{\frac{1}{n}}(e^{sv})\mfrakL_{\frac{1}{n}}(e^{tu})\Big)\\
=&\ \frac{\mathrm{d}^2}{\mathrm{d}t \mathrm{d}s}\bigg|_{t,s=0}\lim_{n\to\infty}\mfrakH_{\frac{1}{n}}
\Big(\mfrakL_{\frac{1}{n}}(e^{tu})\mfrakL_{\frac{1}{n}}(e^{sv})\Big)-\frac{\mathrm{d}^2}{\mathrm{d}t \mathrm{d}s}\bigg|_{t,s=0}\lim_{n\to\infty}\mfrakH_{\frac{1}{n}}
\Big(\mfrakL_{\frac{1}{n}}(e^{sv})\mfrakL_{\frac{1}{n}}(e^{tu})\Big)\\
=&\ \frac{\mathrm{d}^2}{\mathrm{d}t \mathrm{d}s}\bigg|_{t,s=0}\lim_{n\to\infty}\mfrakH_{\frac{1}{n}}
\Big(\mfrakL_{\frac{1}{n}}(e^{tu_n})\mfrakL_{\frac{1}{n}}(e^{sv_n})\Big)-\frac{\mathrm{d}^2}{\mathrm{d}t \mathrm{d}s}\bigg|_{t,s=0}\lim_{n\to\infty}\mfrakH_{\frac{1}{n}}
\Big(\mfrakL_{\frac{1}{n}}(e^{sv_n})\mfrakL_{\frac{1}{n}}(e^{tu_n})\Big)\hspace{1cm}\text{(by Eq.~\eqref{eq:jx})}\\
=&\ \lim_{n\to\infty}\MfrakH_{\frac{1}{n}}
\Big([\MfrakL_{\frac{1}{n}}(u_n),\MfrakL_{\frac{1}{n}}(v_n)]\Big). \qedhere
\end{align*}
}
\end{proof}

\begin{theorem}
Let $(G, \frakB)$ be a Rota-Baxter Lie group with limit-weight $\lim\limits_{n\to\infty}(\mfrakL_{\frac{1}{n}},\mfrakH_{\frac{1}{n}})$ from unital pairs $(\mfrakL_{\frac{1}{n}},\mfrakH_{\frac{1}{n}}), n\in \mathbb{P}$. Let $\frakg = T_e G$ be the Lie algebra of $G$, and let
\vpb
$$
B:=\frakB_{*e},\quad \MfrakL_{\frac{1}{n}}:=(\mfrakL_{\frac{1}{n}})_{*e},\quad \MfrakH_{\frac{1}{n}}:=(\mfrakH_{\frac{1}{n}})_{*e}
\vpb
$$
be the tangent maps at $e$. Then $(\frakg, B)$ is a Rota-Baxter Lie algebra with limit-weight $\lim\limits_{n\to\infty}(\MfrakL_{\frac{1}{n}},\MfrakH_{\frac{1}{n}})$, where the topology on $\frakg$ is induced from $G$. Furthermore, if $(G, \frakB)$ is a Rota-Baxter group with \limwtzero, then $(\frakg, B)$ is a Rota-Baxter Lie algebra with \limwtzero.
\label{thm:tanglim}
\end{theorem}

\begin{proof}
By Lemma~\ref{lemma:groupyr},  $\frakg$ is a $\lim\limits_{n\to \infty} (\MfrakL,\MfrakH)$-Lie algebra. Also we have
\vpb
\small{\begin{align*}
    \lim\limits_{n\to \infty}\MfrakH_{\frac{1}{n}}\Big([B(u),\MfrakL_{\frac{1}{n}}(v)]\Big)= \frac{\mathrm{d}^2}{\mathrm{d}t \mathrm{d}s}\bigg|_{t,s=0}\lim\limits_{n\to \infty}  \mfrakH_{\frac{1}{n}} \Big( \mathfrak{B}(e^{tu}) \mfrakL_{\frac{1}{n}}(e^{sv}) \mathfrak{B}(e^{tu})^{-1} \Big).
\end{align*}}
So we just need to verify Eq.~\eqref{eq:rblim} as follows.
%
{\small
\begin{align*}
\left [ B(u), B(v) \right ]
=&\  \frac{\mathrm{d}^2}{\mathrm{d}t \mathrm{d}s}\bigg|_{t,s=0} e^{tB(u)}e^{sB(v)}e^{-tB(u)}=  \frac{\mathrm{d}^2}{\mathrm{d}t \mathrm{d}s}\bigg|_{t,s=0} \frakB(e^{tu}) \frakB(e^{sv}) \frakB(e^{-tu})\\
=&\  \frac{\mathrm{d}^2}{\mathrm{d}t \mathrm{d}s}\bigg|_{t,s=0} \frakB\bigg(\lim_{n\to \infty} \mfrakH_{\frac{1}{n}}\Big( \mfrakL_{\frac{1}{n}}(e^{tu}) \frakB(e^{tu}) \mfrakL_{\frac{1}{n}}(e^{sv}) \frakB(e^{tu})^{-1} \Big) \bigg)\frakB(e^{-tu})\\
=&\  \frac{\mathrm{d}^2}{\mathrm{d}t \mathrm{d}s}\bigg|_{t,s=0} \frakB\Bigg(\lim_{m\to\infty}\mfrakH_{\frac{1}{n}} \bigg(\mfrakL_{\frac{1}{n}}\Big(\lim_{n\to \infty} \mfrakH_{\frac{1}{n}}\big( \mfrakL_{\frac{1}{n}}(e^{tu}) \frakB(e^{tu}) \mfrakL_{\frac{1}{n}}(e^{sv}) \frakB(e^{tu})^{-1} \big) \Big)\\
&\ \hspace{2cm}\frakB(e^{tu})\frakB(e^{sv})\mfrakL_{\frac{1}{n}}(b)\Big(\frakB(e^{tu})\frakB(e^{sv})\Big)^{-1} \bigg)\Bigg)\\
=&\  \frac{\mathrm{d}^2}{\mathrm{d}t \mathrm{d}s}\bigg|_{t,s=0} \frakB\bigg(\lim_{n\to\infty}\mfrakH_{\frac{1}{n}} \Big( \mfrakL_{\frac{1}{n}}(e^{tu}) \frakB(e^{tu}) \mfrakL_{\frac{1}{n}}(e^{sv}) \frakB(e^{sv}) \mfrakL_{\frac{1}{n}}(e^{-tu}) \frakB(e^{sv})^{-1} \frakB(e^{tu})^{-1}\Big)\bigg)\\
&\hspace{1cm}\text{(by Eq.~(\ref{eq:jx}))}\\
=&\  \lim_{n\to \infty} \frac{\mathrm{d}^2}{\mathrm{d}t \mathrm{d}s}\bigg|_{t,s=0} \frakB\mfrakH_{\frac{1}{n}} \Big( \mfrakL_{\frac{1}{n}}(e^{tu}) \frakB(e^{tu}) \mfrakL_{\frac{1}{n}}(e^{sv}) \frakB(e^{sv}) \mfrakL_{\frac{1}{n}}(e^{-tu}) \frakB(e^{sv})^{-1} \frakB(e^{tu})^{-1}\Big)\\
=&\    B\bigg(\lim_{n\to\infty}\MfrakH_{\frac{1}{n}}\Big([B(u),\MfrakL_{\frac{1}{n}}(v)]+[\MfrakL_{\frac{1}{n}}(u),B(v)]+[\MfrakL_{\frac{1}{n}}(u),\MfrakL_{\frac{1}{n}}(v)]\Big) \bigg).
\end{align*}
}
Further,  if $(G, \frakB)$ is a Rota-Baxter group with \limwtzero, then for any $u\in \frakg$, we have
\begin{equation}
    \lim_{n\to\infty}\MfrakL_{\frac{1}{n}}(u)=\lim_{n\to\infty}\frac{\mathrm{d}}{\mathrm{d}t} \bigg|_{t=0}\mfrakL_{\frac{1}{n}}(e^{tu})=\frac{\mathrm{d}}{\mathrm{d}t} \bigg|_{t=0}\lim_{n\to\infty}\mfrakL_{\frac{1}{n}}(e^{tu})=\frac{\mathrm{d}}{\mathrm{d}t} \bigg|_{t=0}e=0,
\label{eq:tangentzero}
\end{equation}
and so $(\frakg, B)$ is a Rota-Baxter Lie algebra with \limwtzero.
\end{proof}

\subsection{Limit-weighted differential groups and differential Lie algebras}
For differential groups, we bypass the pair-weighted case and go directly to limit-weighted differential groups.

\begin{defn}
Let $G$ be a topological group, and $\mfrakL_{\frac{1}{n}},\mfrakH_{\frac{1}{n}},\frakD:G\rightarrow G$ maps for $n\in \mathbb{P}$.
A $\lim\limits_{n\to \infty}(\mfrakL_{\frac{1}{n}},\mfrakH_{\frac{1}{n}})$-group $G$
is called a {\bf differential group with limit-weight $\lim\limits_{n\to \infty}(\mfrakL_{\frac{1}{n}},\mfrakH_{\frac{1}{n}})$} if for any $a,b\in G$, the limit 
\vpb
$$\lim\limits_{n\to \infty}\mfrakH_{\frac{1}{n}}\left(
a\mfrakL_{\frac{1}{n}}\left(\mathfrak{D}
\left(b\right)\right)a^{-1}\right) 
\vpa
$$
uniquely exists and
\vpa
\begin{equation}
\mathfrak{D}\left(ab\right)=\lim\limits_{n\to \infty}\mfrakH_{\frac{1}{n}}\left(
\mfrakL_{\frac{1}{n}}\left(\mathfrak{D}\left(a\right)\right)a\mfrakL_{\frac{1}{n}}\left(\mathfrak{D}
\left(b\right)\right)a^{-1}\right).
\label{eq:diffgplh1}
\vpb
\end{equation}
If in addition, the limit $\lim\limits_{n\to \infty}\mfrakL_{\frac{1}{n}}(a)$ uniquely exists and equals to $e$ for each $a\in G$, then we call $G$ a {\bf differential group with limit-weight zero}. For either of these two notions, if $G$ is a Lie group and $\mfrakL_{\frac{1}{n}},\mfrakH_{\frac{1}{n}},\frakD$ are smooth maps, then we call $(G,\frakD)$ a {\bf differential Lie group with limit-weight $\lim\limits_{n\to \infty}(\mfrakL_{\frac{1}{n}},\mfrakH_{\frac{1}{n}})$} and respectively, {\bf with limit-weight zero}.
\label{defn:diffgplh1}
\end{defn}

An example of differential group with limit-weight $0$ will be provided in Example~\ref{ex:diffgpzero} based on the notation of a $\mathbb{Q}$-group.
We also define the Lie algebra counterparts of the above notions.

\begin{defn} Let $\frakg$ be a topological Lie algebra and $\MfrakL_{\frac{1}{n}},\MfrakH_{\frac{1}{n}}, \frakD:\frakg\rightarrow \frakg$ linear maps. A $\lim\limits_{n\to\infty}(\MfrakL_{\frac{1}{n}},\MfrakH_{\frac{1}{n}})$-Lie algebra $\frakg$
is called a {\bf differential Lie algebra with limit-weight $\lim\limits_{n\to\infty}(\MfrakL_{\frac{1}{n}},\MfrakH_{\frac{1}{n}})$} if for any $u,v\in \frakg$, the limit
\vpb 
$$\lim\limits_{n\to\infty}\MfrakH_{\frac{1}{n}}\left( [\MfrakL_{\frac{1}{n}} D(u),v]\right)
\vpa
$$
uniquely exists and there is the identity
\vpb
\begin{equation*}
D\left([u,v]\right)=\lim_{n\to\infty}\MfrakH_{\frac{1}{n}}\left( [\MfrakL_{\frac{1}{n}} D(u),v]+[u,\MfrakL_{\frac{1}{n}} D(v)]+[\MfrakL_{\frac{1}{n}} D(u),\MfrakL_{\frac{1}{n}} D(v)]\right).
\label{eq:difdb}
\vpb
\end{equation*}
Further, if for each $u\in \frakg$, $\lim\limits_{n\to\infty}\MfrakL_{\frac{1}{n}}(u)=0$, we call $(\frakg,D)$ a {\bf  differential Lie algebra with \limwtzero}.
\label{defn:difdb}
\end{defn}

\delete{
\begin{remark}
Similar to Remark~\ref{rk:weight 0}, a diferential Lie algebra with weight $\lim\limits_{n\to\infty}(\MfrakL_{\frac{1}{n}},\MfrakH_{\frac{1}{n}})$ can be reduced to a usual different Lie algebra with weight $0$.
\end{remark}
}

\begin{theorem}
Let $(G, \frakD)$ be a differential Lie group with limit-weight $\lim\limits_{n\to\infty}(
\mfrakL_{\frac{1}{n}},\mfrakH_{\frac{1}{n}})$ for which $(\mfrakL_{\frac{1}{n}},\mfrakH_{\frac{1}{n}}), n\in \PP,$ are unital pairs. Let $\frakg = T_e G$ be the Lie algebra of $G$, and let the linear maps
\vpb
$$D=\frakD_{*e},\quad\MfrakL_{\frac{1}{n}}=(\mfrakL_{\frac{1}{n}})_{*e},\quad \MfrakH_{\frac{1}{n}}=(\mfrakH_{\frac{1}{n}})_{*e}
\vpb
$$
on $\frakg$ be the tangent maps of $\frakD$, $\mfrakL$ and $\mfrakH$ at the identity $e$, respectively.
Then $\left(\mathfrak{g},D\right)$ is a differential Lie algebra with limit-weight $\lim\limits_{n\to\infty}(\MfrakL_{\frac{1}{n}},\MfrakH_{\frac{1}{n}})$, where the topology on $\frakg$ is induced from $G$.  If $(G, \frakB)$ is a differential group with limit-weight zero, then $(\frakg, B)$ is a differential Lie algebra with \limwtzero.
\label{thm:diffgLie1}
\end{theorem}

\begin{proof}
The first claim follows from 
\small{\begin{align}
    \lim\limits_{n\to\infty}\MfrakH_{\frac{1}{n}}\left( [\MfrakL_{\frac{1}{n}} D(u),v]\right)=\left.\dfrac{\dd^2}{\dd t\dd s}\right|_{t,s=0}\lim\limits_{n\to \infty}\mfrakH_{\frac{1}{n}}\left(
e^{tu}\mfrakL_{\frac{1}{n}}\left(\mathfrak{D}
\left(e^{sv}\right)\right)e^{-tu}\right)
\end{align}}
and
{\small
\begin{align*}
D\left[u,v\right] =&\ \left.\dfrac{\dd^2}{\dd t\dd s}\right|_{t,s=0}\frakD\left(e^{tu}e^{sv}e^{-tu}\right)\\
=&\ \left.\dfrac{\dd^2}{\dd t\dd s}\right|_{t,s=0} \lim_{n\to\infty}\mfrakH_{\frac{1}{n}}\bigg( \mfrakL_{\frac{1}{n}}\frakD(e^{tu})e^{tu}
\mfrakL_{\frac{1}{n}}\frakD(e^{sv})e^{sv}\mfrakL_{\frac{1}{n}}\frakD\left(e^{-tu}\right)e^{-sv}e^{-tu}
\Big)  \quad \text{(by Eq.~\eqref{eq:jx})}\\
=&\ \lim_{n\to\infty} (\mfrakH_{\frac{1}{n}})_{*e} \Big( \left.\dfrac{\dd^2}{\dd t\dd s}\right|_{t,s=0}e^{sv}\mfrakL_{\frac{1}{n}}\frakD(e^{-tu})e^{-sv}+\left.\dfrac{\dd^2}{\dd t\dd s}\right|_{t,s=0} e^{tu}\mfrakL_{\frac{1}{n}}\frakD(e^{sv})e^{-tu}\\
&\ +\left.\dfrac{\dd^2}{\dd t\dd s}\right|_{t,s=0} \mfrakL_{\frac{1}{n}}\frakD(e^{tu})\mfrakL_{\frac{1}{n}}\frakD(e^{sv})\mfrakL_{\frac{1}{n}}\frakD(e^{-tu})\bigg)  \\
=&\ \lim_{n\to\infty}\MfrakH_{\frac{1}{n}}\left( [\MfrakL_{\frac{1}{n}} D(u),v]+[u,\MfrakL_{\frac{1}{n}} D(v)]+[\MfrakL_{\frac{1}{n}} D(u),\MfrakL_{\frac{1}{n}} D(v)]\right), \quad u,v\in \frakg.
\end{align*}
}
Further,  if $(G, \frakB)$ is a differential group with limit-weight zero, then by Eq.~\eqref{eq:tangentzero}, $(\frakg, D)$ is a differential Lie algebra with \limwtzero.
\end{proof}

\subsection{Integration and derivation of group-values functions}
\label{ss:intgroup}
As is well-known, Rota-Baxter operators and differential operators with weight zero on algebras are generalizations of integral operators and derivations in analysis.
This subsection is devoted to introducing the integration and derivation for maps with values in groups.

\begin{defn}
Let $G$ be a topological group and $\mfrakL_{\lambda},\mfrakH_{\lambda}$:$G\rightarrow G$ for $\lambda\in (0,1)$ be maps. For a subset $S$ of $G$, define a family of sets $S_{\lambda}:=S\cap \mfrakH_{\lambda}^{-1}(S)$.
\begin{enumerate}
\item We call $S$ a {\bf $(\mfrakL_{\lambda},\mfrakH_{\lambda})$-subset of  $G$} if $(\mfrakL_{\lambda}\mfrakH_{\lambda})|_{S_{\lambda}}=\id_{S_\lambda}$ for $\lambda\in (0,1)$.

\item If in addition, $\lim\limits_{\lambda \to 0}\mfrakL_{\lambda}(a)$ uniquely exists and equal to $e$ for all $a\in S$ and $\lambda\in (0,1)$, then we call $S$ a {\bf $(\mfrakL_{\lambda},\mfrakH_{\lambda})$-subset of $G$ with limit-weight zero}.
For notational convenience, we also define $\mfrakL_0:G\to G, a\mapsto e$, so as to have $\lim\limits_{\lambda \to 0}\mfrakL_{\lambda}(a)=\mfrakL_0(a)$ for each $a\in S$.
%
\end{enumerate}
\mlabel{defn:ppgroup}
\end{defn}

\begin{remark}
In particular, if $S_\lambda=S$ and $(\mfrakL_{\lambda}\mfrakH_{\lambda})|_{S}=\id_S$, then $S$ is a  $(\mfrakL_{\lambda},\mfrakH_{\lambda})$-subset. If in addition $S=G$, then $G$ is a  $(\mfrakL_{\lambda},\mfrakH_{\lambda})$-group.
\end{remark}

We expose an example of a $(\mfrakL_{\lambda},\mfrakH_{\lambda})$-subset.

\begin{exam}
Let $G$ be a Lie group  with $\frakg$ as its Lie algebra. Let $W$ be a subset of $\frakg$ such that for any $t\in (0,1)$ and $u\in W$, we have $tu\in W$, and such that $\exp|_W$ is injective.
\delete{
for any linearly independent $u,v \in W $ and for any $s,t\in (0,1)$,
$\exp(tu)\neq \exp(sv)$ and $tu\in W$
.
}
Define $S:=\exp(W)$ and two maps
\vpb
$$\mfrakL_{\lambda}:G\rightarrow G, \quad a\mapsto\Bigg\{\begin{array}{cc}
\exp(\lambda u), & \text{if }a\in S,\, a=\exp(u),  \\
         I, & \text{otherwise,}
\end{array}$$
$$\mfrakH_{\lambda}:G\rightarrow G, \quad a\mapsto \Bigg\{
\begin{array}{cc}
\exp(\frac{1}{\lambda} u),&  \text{if } a\in S,\, a=\exp(u),\, \exp(\frac{1}{\lambda} u)\in S,\\
         I, & \text{otherwise.}
\end{array}  $$
For each $a=\exp(u)$ in $S_{\lambda}$, we have $\frac{1}{\lambda}u \in W$ and  \vpb
$$\mfrakL_{\lambda}\mfrakH_{\lambda}(a)=\mfrakL_{\lambda}
\Big(\exp(\frac{1}{\lambda}u)\Big)=\exp(u)=a.
\vpb$$
For each $b=\exp(v)\in S$, we get
\vpb
$$\lim\limits_{n\to\infty}\mfrakL_{\frac{1}{n}}(b)=
\lim\limits_{n\to\infty}\exp(\frac{1}{n}v)=I.$$
Thus $S$ is a $(\mfrakL_{\lambda},\mfrakH_{\lambda})$-subset of $G$.
\end{exam}

Another example will be given in Example~\ref{ex:ljfz}.
We are going to define the integral in a $(\mfrakL_{\lambda},\mfrakH_{\lambda})$-subset in $S$ with limit-weight zero.
As usual, a {\bf partition} $P$ of a closed interval $[\ell,m]\subseteq \mathbb{R}$ is a finite set of points $\{t_0,t_1,\cdots, t_n\}$ such that $\ell= t_0\leq t_1 \leq\cdots \leq t_n=m$. Denote $\triangle_i:= t_i-t_{i-1}$ and $|P|:= \max\limits_{1\leqslant i\leqslant n}\triangle_i$.

\begin{defn}
Let $G$ be a topological group and $S$ a $(\mfrakL_{\lambda},\mfrakH_{\lambda})$-subset of $G$ with weight zero. A map $a:\mathbb{R}\to S$ is called {\bf $(\mfrakL_{\lambda},\mfrakH_{\lambda})$-integrable} on a closed interval $[\ell,m]\subset \RR$ if the limit
\vpb$$\lim\limits_{|P|\to 0 }\prod\limits_{k=1}^{n}\mfrakL_{{\triangle_{n+1-k}}}\Big(a(\xi_{n+1-k})\Big)$$
uniquely exists in $S$ for arbitrary choices of $\xi_k \in [t_{k-1},t_{k}]$. In this case, we denote the limit by \vpa
$$\int_{\ell}^{m}a(t)d_S\,t:=\int_{\ell}^{m}a(t)
d_S^{(\mfrakL_{\lambda},\mfrakH_{\lambda})}t
\vpa$$
and call it the {\bf $(\mfrakL_{\lambda},\mfrakH_{\lambda})$-integral} of the map $a$ on $[\ell,m]$.
\label{defn:IntLH}
\end{defn}

Before stating the main theorem in this subsection, we need to introduce more notions.

\begin{defn}
\label{defn:intsnow}
Let $G$ be a topological group and $S$ a $(\mfrakL_{\lambda},\mfrakH_{\lambda})$-subset of $G$ with weight zero.
\begin{enumerate}
\item Let $a_n:[0,x]\to S, n\in \PP,$ be a sequence of maps such that $\lim\limits_{n\to \infty} a_n(t)$ uniquely exists for each $t\in [0,x]$. We call $\lim\limits_{n\to \infty} a_n$ {\bf \compint on $[0,x]$} if $\lim\limits_{n\to\infty}a_{n}:[0,x]\rightarrow S$ is a $(\mfrakL_{\lambda},\mfrakH_{\lambda})$-integrable map, the limit $\lim\limits_{n \to \infty }\prod\limits_{k=1}^{n}
\mfrakL_{\frac{1}{n}}\bigg(a_n\Big({\frac{n+1-k}{n}}x\Big)\bigg)$ uniquely exists in $S$ and the integral can be obtained as
\vpb
\begin{align}
\int_{0}^{x}\lim_{n\to\infty}a_{n}(t)d_{S}t
=\lim\limits_{n \to \infty }\prod\limits_{k=1}^{n}
\mfrakL_{\frac{1}{n}}\bigg(a_n\Big({\frac{n+1-k}{n}}x\Big)\bigg).
\label{eq:int-commutative}
\vpc
\end{align}
\label{it:int-commutative}
\item We call $S$ a {\bf \compgroup subset} if for each $a_{n,k},b_{n,k},c_{n,k,m},d_{n,k,m}\in S$ with
$\lim\limits_{m\to\infty}c_{n,k,m}=c_{n,k}$
and $\lim\limits_{m\to\infty}d_{n,k,m}=c_{n,k}^{-1}$, the limits
\vpb
$$    \lim\limits_{n \to \infty }\prod\limits_{k=1}^{n}\mfrakL_{\frac{1}{n}}(a_{n,k})c_{n,k}
\mfrakL_{\frac{1}{n}}(b_{n,k})c_{n,k}^{-1} \,\text{ and }\, \lim\limits_{n \to \infty }\prod\limits_{k=1}^{n}\mfrakL_{\frac{1}{n}}(a_{n,k})c_{n,k,n}\mfrakL_{\frac{1}{n}}(b_{n,k})d_{n,k,n}
\vpb
$$
uniquely exist in $S$ and coincide:
\vpb
\begin{align}
\label{eq:int-algebraic}
    \lim\limits_{n \to \infty }\prod\limits_{k=1}^{n}\mfrakL_{\frac{1}{n}}(a_{n,k})c_{n,k}\mfrakL_{\frac{1}{n}}(b_{n,k})c_{n,k}^{-1}=\lim\limits_{n \to \infty }\prod\limits_{k=1}^{n}\mfrakL_{\frac{1}{n}}(a_{n,k})c_{n,k,n}\mfrakL_{\frac{1}{n}}(b_{n,k})d_{n,k,n}.
\end{align}
\end{enumerate}
\end{defn}

The integral defined above satisfies the integration-by-parts formula in the following sense.

\begin{theorem}
\label{thm:HIfml}
Let $G$ be a topological group, and $S$ a $(\mfrakL_{\lambda},\mfrakH_{\lambda})$-subset of $G$ with weight zero and also a \compgroup subset. Let $a,b:[0,x]\to S$ be $(\mfrakL_{\lambda},\mfrakH_{\lambda})$-integrable maps on $[0,t]$ for each $t\in [0,x]$. If for each $\lambda\in (0,1)$ and $t\in [0,x]$, we have \begin{equation}
\mfrakL_{\lambda}\Big(a(t)\Big)\ad_{\Big(\int_{0}^{t}a(s)
d_S^{(\mfrakL_{\lambda},\mfrakH_{\lambda})}s\Big)} \mfrakL_{\lambda}\Big(b(t)\Big)\in S_{\lambda}
\mlabel{eq:inslam}
\end{equation} and the limit
\[
[0, x]\rightarrow S, \quad t\mapsto \lim\limits_{n\to \infty}\mfrakH_{\frac{1}{n}}\bigg(\mfrakL_{\frac{1}{n}}\Big(a(t)\Big)\ad_{\Big(\int_{0}^{t}a(s)d_S^{(\mfrakL_{\lambda},\mfrakH_{\lambda})}s\Big)} \mfrakL_{\frac{1}{n}}\Big(b(t)\Big)\bigg)
\]
is \compint on $[0,x]$, then the relation for a Rota-Baxter operator with limit-weight in Eq.~\eqref{eq:limitwt} holds:
    \begin{align*}
    \bigg(\int_{0}^{x}a(t)d_S^{(\mfrakL_{\lambda},\mfrakH_{\lambda})}t\bigg)\bigg(\int_{0}^{x}b(t)d_S^{(\mfrakL_{\lambda},\mfrakH_{\lambda})}t\bigg)=\int_{0}^{x}\lim_{n\to \infty}\mfrakH_{\frac{1}{n}}\bigg(\mfrakL_{\frac{1}{n}}\Big(a(t)\Big)\ad_{\Big(\int_{0}^{t}a(s)d_S^{(\mfrakL_{\lambda},\mfrakH_{\lambda})}s\Big)} \mfrakL_{\frac{1}{n}}\Big(b(t)\Big)\bigg)d_S^{(\mfrakL_{\lambda},\mfrakH_{\lambda})}t.
    \end{align*}
\end{theorem}
Note that here we cannot yet claim to have a Rota-Baxter operator with limit-weight zero since there still needs to be a group that is closed under the operation
$$ a \mapsto \int_{0}^{x}a(t)d_S^{(\mfrakL_{\lambda},\mfrakH_{\lambda})}t.
$$

\begin{proof}
We have
{\small
\begin{align*}
&\int_{0}^{x}\lim_{n\to \infty}\mfrakH_{\frac{1}{n}}\bigg(\mfrakL_{\frac{1}{n}}\Big(a(t)\Big)
\Big(\int_{0}^{t}a(s)d_S^{(\mfrakL_{\lambda},\mfrakH_{\lambda})}s\Big) \mfrakL_{\frac{1}{n}}\Big(a(t)\Big)\Big(\int_{0}^{t}a(s)d_S^{(\mfrakL_{\lambda},
\mfrakH_{\lambda})}s\Big)^{-1}\bigg)d_S^{(\mfrakL_{\lambda},\mfrakH_{\lambda})}t\\
=&\   \lim_{\substack{m=n\to \infty}}\prod_{k=1}^{m}\Bigg[\mfrakL_{\frac{1}{m}}\mfrakH_{\frac{1}{n}}\Bigg(\bigg(\mfrakL_{\frac{1}{n}}\Big(a\bigg(\frac{m+1-k
}{m}x\bigg)\Big)\Big(\int_{0}^{\frac{m+1-k
}{m}x}\hspace{-.7cm}a(s)d_S^{(\mfrakL_{\lambda},\mfrakH_{\lambda})}s\Big) \mfrakL_{\frac{1}{n}}\Big(b\bigg(\frac{m+1-k
}{m}x\bigg)\Big)\Big(\int_{0}^{\frac{m+1-k
}{m}x}\hspace{-.7cm} a(s)d_S^{(\mfrakL_{\lambda},\mfrakH_{\lambda})}s\Big)^{-1}\bigg)\Bigg)\Bigg]\\
& \hspace{7cm} \text{(by Eq.~\eqref{eq:int-commutative}) }\\
=&\   \lim_{\substack{n\to\infty}}\prod_{k=1}^{n}\Bigg[\mfrakL_{\frac{1}{n}}\Big(a\bigg(\frac{n+1-k
}{n}x\bigg)\Big)\Big(\int_{0}^{\frac{n+1-k
}{n}x}a(s)d_S^{(\mfrakL_{\lambda},\mfrakH_{\lambda})}s\Big) \mfrakL_{\frac{1}{n}}\Big(b\bigg(\frac{n+1-k
}{n}x\bigg)\Big)\Big(\int_{0}^{\frac{n+1-k
        }{n}x}a(s)d_S^{(\mfrakL_{\lambda},\mfrakH_{\lambda})}s\Big)^{-1}\Bigg].\\
 & \hspace{7cm} \text{(by Eq.~\eqref{eq:inslam}) }
\end{align*}
}
We will use a special partition for the integrals to convert the above product into a telescopic type product, so that most factors can be canceled.
Thanks to Definition~\ref{defn:IntLH}, for each $k$, we have
$${\small \int_{0}^{\frac{n+1-k
}{n}x}a(s)d_S^{(\mfrakL_{\lambda},\mfrakH_{\lambda})}s=\lim_{m\to\infty}\prod_{s=1}^{m}\mfrakL_{\frac{n+1-k}{nm}}
\bigg(a\Big(\frac{n+1-k
}{n}\frac{m+1-s}{m}x\Big)\bigg)
}
$$
and
$$
{\small \Big(\int_{0}^{\frac{n+1-k
        }{n}x}a(s)d_S^{(\mfrakL_{\lambda},\mfrakH_{\lambda})}s\Big)^{-1}=\Bigg(\lim_{m\to\infty}\mfrakL_{\frac{1}{n}}\bigg(a\Big(\frac{n-k}{n}x\Big)\bigg)
\prod_{s=1}^{m}\mfrakL_{\frac{n-k}{nm}}\bigg(a\Big(\frac{n-k
}{n}\frac{m+1-s}{m}x\Big)\bigg)\Bigg)^{-1},
}$$
which implies
{\small
\begin{align*}
&\int_{0}^{x}\lim_{n\to \infty}\mfrakH_{\frac{1}{n}}\bigg(\mfrakL_{\frac{1}{n}}\Big(a(t)\Big)
\Big(\int_{0}^{t}a(s)d_S^{(\mfrakL_{\lambda},\mfrakH_{\lambda})}s\Big) \mfrakL_{\frac{1}{n}}\Big(a(t)\Big)\Big(\int_{0}^{t}a(s)d_S^{(\mfrakL_{\lambda},
\mfrakH_{\lambda})}s\Big)^{-1}\bigg)d_S^{(\mfrakL_{\lambda},\mfrakH_{\lambda})}t\\
=&\ \lim_{\substack{n\to\infty}}\prod_{k=1}^{n}\Bigg[\mfrakL_{\frac{1}{n}}
\Big(a\bigg(\frac{n+1-k
}{n}x\bigg)\Big)\bigg(\lim_{m\to\infty}\prod_{s=1}^{m}\mfrakL_{\frac{n+1-k}{nm}}
\Big(a\bigg(\frac{n+1-k
}{n}\frac{m+1-s}{m}x\bigg)\Big)\bigg)\\
&\qquad \qquad \mfrakL_{\frac{1}{n}}\Big(b\bigg(\frac{n+1-k
}{n}x\bigg)\Big)\bigg(\lim_{m\to\infty}\mfrakL_{\frac{1}{n}}\Big(a\bigg(\frac{n-k}{n}x\bigg)\Big)
\prod_{s=1}^{m}\mfrakL_{\frac{n-k}{nm}}\Big(a\bigg(\frac{n-k
}{n}\frac{m+1-s}{m}x\bigg)\Big)\bigg)^{-1}\Bigg]\\
=&\ \lim_{\substack{n=m\to\infty}}\prod_{k=1}^{n}\Bigg[\mfrakL_{\frac{1}{n}}
\Big(a\bigg(\frac{n+1-k
}{n}x\bigg)\Big)\Big(\prod_{s=1}^{m}\mfrakL_{\frac{n+1-k}{nm}}\Bigg(a\bigg(\frac{n+1-k
}{n}\frac{m+1-s}{m}x\bigg)\Bigg)\\
&\qquad \qquad \ \ \ \mfrakL_{\frac{1}{n}}\Big(b\bigg(\frac{n+1-k
}{n}x\bigg)\Big)\bigg(\mfrakL_{\frac{1}{n}}\Big(a\bigg(\frac{n-k}{n}x\bigg)\Big)\prod_{s=1}^{m}
\mfrakL_{\frac{n-k}{nm}}\Big(a\bigg(\frac{n-k
}{n}\frac{m+1-s}{m}x\bigg)\Big)\bigg)^{-1}\Bigg]
\ \ \  \text{(by Eq.~\eqref{eq:int-algebraic})}\\
=&\ \lim_{\substack{n\to\infty}}\prod_{k=1}^{n}\Bigg[\mfrakL_{\frac{1}{n}}\Big(a(\frac{n+1-k
}{n}x)\Big)\bigg(\prod_{s=1}^{n}\mfrakL_{\frac{n+1-k}{n^2}}\Big(a\bigg(\frac{n+1-k
}{n}\frac{n+1-s}{n}x\bigg)\Big)\bigg)\\
&\qquad\qquad \mfrakL_{\frac{1}{n}}\Big(b\bigg(\frac{n+1-k
}{n}x\bigg)\Big)\bigg(\mfrakL_{\frac{1}{n}}\Big(a\bigg(\frac{n-k}{n}x\bigg)\Big)\prod_{s=1}^{n}
\mfrakL_{\frac{n-k}{n^2}}\Big(a\bigg(\frac{n-k
}{n}\frac{n+1-s}{n}x\bigg)\Big)\bigg)^{-1}\Bigg]\\
=&\ \lim_{\substack{n\to\infty}}\mfrakL_{\frac{1}{n}}\Big(a(x)\Big)\bigg(\prod_{s=1}^{n}
\mfrakL_{\frac{1}{n}}\Big(a\bigg(\frac{n+1-s}{n}x\bigg)\Big)\bigg)\prod_{k=1}^{n}\mfrakL_{\frac{1}{n}}
\Big(b\bigg(\frac{n+1-k}{n}x\bigg)\Big)\bigg(\mfrakL_{\frac{1}{n}}\Big(a(0)\Big)\prod_{s=1}^{n}\mfrakL_0
\Big(a(0)\Big)\bigg)^{-1}\\
&\text{\bigg(by $\bigg(\mfrakL_{\frac{1}{n}}\Big(a\bigg(\frac{n-k}{n}x\bigg)\Big)\prod_{s=1}^{n}
\mfrakL_{\frac{n-k}{n^2}}\Big(a\bigg(\frac{n-k
}{n}\frac{n+1-s}{n}\bigg)x\Big)\bigg)^{-1}\bigg(\mfrakL_{\frac{1}{n}}\Big(a\bigg(\frac{n-k}{n}x\bigg)\Big)\prod_{s=1}^{n}
\mfrakL_{\frac{n-k}{n^2}}\Big(a\bigg(\frac{n-k
}{n}\frac{n+1-s}{n}x\bigg)\Big)\bigg)=e$\bigg)}\\
=&\  \bigg(\int_{0}^{x}a(t)d_S^{(\mfrakL_{\lambda},\mfrakH_{\lambda})}t\bigg)\bigg(\int_{0}^{x}b(t)
d_S^{(\mfrakL_{\lambda},\mfrakH_{\lambda})}t\bigg).
\end{align*}
}
Here the last step employs 
\vpb
$$
\hspace{1.5cm}
\lim_{\substack{n\to\infty}}\mfrakL_{\frac{1}{n}}\Big(f(x)\Big)=e\ \text{ and }\ 
\int_{0}^{x}f(t)d_S^{(\mfrakL_{\lambda},\mfrakH_{\lambda})}t= \lim\limits_{\substack{n\to\infty}}\prod_{s=1}^{n}\mfrakL_{\frac{1}{n}}
\bigg(f\bigg(\frac{n+1-s}{n}x\bigg)\bigg). \hspace{2cm} \qedhere$$
\end{proof}

Now we turn to the differential analog of Definition~\ref{defn:IntLH}.
Taking $\Psi_\lambda(a)$ as an abstraction of $a^{\frac{1}{\lambda}}$, we propose the following multiplicative notion of derivatives. 

\begin{defn}
Let $G$ be a topological group and $S$ a $(\mfrakL_{\lambda},\mfrakH_{\lambda})$-subset of $G$ and $a:\mathbb{R}\to S$ a map. We say that the map $a$ is {\bf $(\mfrakL_{\lambda},\mfrakH_{\lambda})$-differentiable} at $x$ if, for each $\lambda\in (0,1)$ and $x\in \mathbb{R}$, the element $a(x+\lambda)a(x)^{-1}$ is in $S_{\lambda}$ and $\lim\limits_{\lambda \to 0 }\mfrakH_{\lambda}(a(x+\lambda)a(x)^{-1})$  uniquely exists in $S$. In this case, for each $x \in \mathbb{R}$, we denote the limit by 
\vpb
$$\frac{d_S}{d_S x}\Big(a(x)\Big):=\frac{d_S}{d_S^{(\mfrakL_{\lambda},\mfrakH_{\lambda})}x}\Big(a(x)\Big)
\vpb$$ 
and call it the {\bf $(\mfrakL_{\lambda},\mfrakH_{\lambda})$-derivative of $a$ at $x$}.
\label{defn:DfLH}
\end{defn}

The above derivative satisfies the Leibniz rule in the following sense.

\begin{theorem}
Let $G$ be a topological group and $S$ a $(\mfrakL_{\lambda},\mfrakH_{\lambda})$-subset of $G$. Suppose that the maps $a,b :\mathbb{R}\to S$ are $(\mfrakL_{\lambda},\mfrakH_{\lambda})$-differentiable at each $x\in \RR$. If for each $\lambda\in (0,1)$, $u,v\in S$ and $x\in \mathbb{R}$ the element $\mfrakL_{\lambda} (u)a(x)\mfrakL_{\lambda}(v)a(x)^{-1}$ is in $S_{\lambda}$ and
\vpb
\begin{align*}
F_S:S\times S \rightarrow S,\quad
  (u,v) \mapsto  \lim\limits_{n\to \infty}\mfrakH_{\frac{1}{n}}
\bigg(\mfrakL_{\frac{1}{n}}
(u)
a(x)\mfrakL_{\frac{1}{n}}(v)a(x)^{-1}\bigg)
\vpb
\end{align*}
is a \complim map for each $x\in \mathbb{R}$, then Eq.~\eqref{eq:diffgplh1} holds:
\begin{align*}
        \frac{d_S}{d_S^{(\mfrakL_{\lambda},\mfrakH_{\lambda})}x}\Big(a(x)b(x)\Big)=\lim_{n\to \infty}\mfrakH_{\frac{1}{n}}\bigg[\mfrakL_{\frac{1}{n}}\bigg(\frac{d_S}{d_S^{(\mfrakL_{\lambda},\mfrakH_{\lambda})}x}\Big(a(x)\Big)\bigg)a(x)\mfrakL_{\frac{1}{n}}\bigg(\frac{d_S}{d_S^{(\mfrakL_{\lambda},\mfrakH_{\lambda})}x}\Big(b(x)\Big)\bigg)a(x)^{-1}\bigg].
    \end{align*}
\label{thm:HDfml}
\end{theorem}

\begin{remark}
As with the remark made after Theorem~\ref{thm:HIfml},
we cannot yet claim to have a differential operator with limit-weight zero since there still needs to identify a group that is closed under the assignment
\vpb
$$ a \mapsto \frac{d_S}{d_S^{(\mfrakL_{\lambda},\mfrakH_{\lambda})}x}\Big(a(x)\Big).
\vpa
$$
It would be interesting to find the closedness conditions for these operators. It would also be interesting to see whether a type of First Fundamental Theorem of calculus holds for the operators $\frac{d_S}{d_S^{(\mfrakL_{\lambda},\mfrakH_{\lambda})}x}\Big(a(x)\Big)$ and $\int_{0}^{x}a(t)d_S^{(\mfrakL_{\lambda},\mfrakH_{\lambda})}t$. See Proposition~\ref{pp:fftc} for a related result.
\end{remark}

\begin{proof}[Proof of Theorem~\ref{thm:HDfml}]
We have
\vpb
{\small
\begin{align*}
&\ \lim_{n\to \infty}\mfrakH_{\frac{1}{n}}\bigg[\mfrakL_{\frac{1}{n}}\bigg(\frac{d_S}
{d_S^{(\mfrakL_{\lambda},\mfrakH_{\lambda})}x}\Big(a(x)\Big)\bigg)a(x)
\mfrakL_{\frac{1}{n}}\bigg(\frac{d_S}{d_S^{(\mfrakL_{\lambda},\mfrakH_{\lambda})}x}\Big(b(x)
\Big)\bigg)a(x)^{-1}\bigg]\\
=&\ \lim_{\substack{n \to \infty }}\mfrakH_{\frac{1}{n}}\bigg[\mfrakL_{\frac{1}{n}}\bigg(\lim_{\substack{ m \to \infty }}\mfrakH_{\frac{1}{m}}\Big(a\bigg(x+\frac{1}{m}\bigg)a(x)^{-1}\Big)\bigg)a(x)
\mfrakL_{\frac{1}{n}}\bigg(\lim_{\substack{ m \to \infty }}\mfrakH_{\frac{1}{m}}\bigg(b\Big(x+\frac{1}{m}\Big)b(x)^{-1}\bigg)\bigg)a(x)^{-1}\bigg]\\
&\hspace{5cm}\text{\bigg(by the definition of $\frac{d_S}{d_S^{(\mfrakL_{\lambda},\mfrakH_{\lambda})}x}\Big(a(x)\Big)$\bigg)}\\
=&\ \lim_{\substack{n \to \infty}}\mfrakH_{\frac{1}{n}}\bigg[\mfrakL_{\frac{1}{n}}
\bigg(\mfrakH_{\frac{1}{n}}\Big(a\bigg(x+\frac{1}{n}\bigg)a(x)^{-1}\Big)\bigg)
a(x)\mfrakL_{\frac{1}{n}}\bigg(\mfrakH_{\frac{1}{n}}\Big(b\bigg(x+\frac{1}{n}\bigg)b(x)^{-1}\Big)
\bigg)a(x)^{-1}\bigg]\\
&\hspace{5cm}\text{(by the \complim property of $F_S$)}\\
=&\ \lim_{\substack{n \to \infty}}\mfrakH_{\frac{1}{n}}\bigg[a(x+\frac{1}{n})a(x)^{-1}
a(x)b(x+\frac{1}{n})b(x)^{-1}a(x)^{-1}\bigg]\\
=& \lim_{\substack{n \to \infty}}\mfrakH_{\frac{1}{n}}
\bigg[a\bigg(x+\frac{1}{n}\bigg)b\bigg(x+\frac{1}{n}\bigg)
b(x)^{-1}a(x)^{-1}\bigg]\\
=&\ \lim_{\substack{n \to \infty}}
\mfrakH_{\frac{1}{n}}\bigg[a\bigg(x+\frac{1}{n}\bigg)
b\bigg(x+\frac{1}{n}\bigg)
\Big(a(x)b(x)\Big)^{-1}\bigg] \\
=& \frac{d_S}{d_S^{(\mfrakL_{\lambda},\mfrakH_{\lambda})}x}\Big(a(x)b(x)\Big). \qedhere
\vpe
\end{align*}
}
\vpd
\end{proof}
\vspace{-.5cm}

\section{Rota-Baxter $\mathbb{Q}$-groups and differential $\mathbb{Q}$-groups}
\label{sec:qrgroup}

To provide both motivations and applications of the above general results, we now focus on a special case of a Rota-Baxter group with weight $\lim\limits_{n\to\infty}(\mfrakL_{\frac{1}{n}},\mfrakH_{\frac{1}{n}})$, namely a Rota-Baxter $\mathbb{Q}$-group. Here the constructions can be made explicitly and are directly related to the classical notions of integrations.

\subsection{Rota-Baxter $\QQ$-groups with weight zero}

Let us begin with the following concept.

\begin{defn}
\label{defn:rgroup}
\begin{enumerate}
\item A topological group $G$ is called a {\bf $\mathbb{Q}$-group} if for each $n\in \mathbb{Z}\backslash\{0\}$, the map
$\pown:G \rightarrow G, \, a\mapsto a^{n}$
is bijective and if the limit $\lim\limits_{n\to\infty}a^{\frac{1}{n}}$ uniquely exists and equals to $e$ for each $a\in G$.
Further, if $G$ is a Lie group, then it is called a {\bf Lie $\mathbb{Q}$-group.}
\item A topological group $G$ is called an {\bf $\mathbb{R}$-group} if for each $a\in G$, there is uniquely a one-parameter subgroup $K_a(\RR)$ defined by a continuous map
$K_a: \RR \to G$
such that $K_a(0)=e$, $K_a(1)=a$ and $K_a(s+t)=K_a(s)K_a(t)$.
For a fixed $r\in \RR$, we define the {\bf $r$-th power map}
\vpa
\begin{equation*}
	P_r:G\to G, \quad a\mapsto a^{r} :=K_a(r).
	\label{eq:power}
\vpa
\end{equation*}
\end{enumerate}
\end{defn}

\begin{lemma}
Let $G$ be an $\RR$-group. Then for each $r$ in $\RR\backslash\{0\}$, $P_{\frac{1}{r}}P_r=\id_G$ and an $\RR$-group is a $\mathbb{Q}$-group and a $(P_{\lambda},P_{\frac{1}{\lambda}})$-group.
\end{lemma}

\begin{proof}
Notice that for each $a\in G,r\in \mathbb{R}\backslash \{0\}$, the two maps 
\vpa
$$f:x\mapsto K_a(rx), \quad g:x\mapsto K_{a^r}(x),
\vpa
$$ 
are both one-parameter subgroup $K_a(\RR)$ satisfying $f(0)=g(0)=e$ and $f(1)=g(1)=a^{r}$. So by the uniqueness of the one-parameter subgroup, we have $f=g$. Thus  
\vpb
\begin{equation}
(a^{r})^{\frac{1}{r}}=K_{a^r}\bigg(\frac{1}{r}x\bigg)=K_{a}\bigg(r\frac{1}{r}\bigg)=a,
\label{eq:opg}
\vpb
\end{equation}
which implies $P_{\frac{1}{r}}P_r=\id_G$ and $G$ is a $(P_{\lambda},P_{\frac{1}{\lambda}})$-group. Now
 taking respectively $r:=n$ and $r:=\frac{1}{n}$ in Eq.~\eqref{eq:opg} and noting that $\lim\limits_{n\to\infty}a^{\frac{1}{n}}=K_a(0)=e$, we conclude that an $\mathbb{R}$-group is a $\mathbb{Q}$-group.
\end{proof}

We expose an example.

\begin{exam} \cite[Theorem 1.127]{Kn}
Let $G$ be a simply connected nilpotent analytic group with  Lie algebra $g$.
Then the exponential map $\exp$ is a diffeomorphism from $\frakg$ onto $G$. So for each $r\in \mathbb{R}$ and $e^{u}\in G$, we can define the $r$ index map of $e^{u}$ as $(e^{u})^{r}:=e^{ru}$. The inverse of $\pown:e^u\mapsto e^{nu}$ is $\pown^{-1}:e^u\mapsto e^{\frac{u}{n}}$ because of the bijectivity of $\exp:\frakg \rightarrow G$. Thus $G$ is a $\mathbb{Q}$-group. For each $e^u\in G$, there is a continuous map $K_{e^{u}}: \mathbb{R}\rightarrow G$, $K_{e^{u}}(r)=e^{r u}$, where 
\vpa
$$K_{e^{u}}(s+t)=e^{(s+t)u}=e^{su+tu}=e^{su}e^{tu}=K_{e^{u}}(s) K_{e^{u}}(t).
\vpa
$$
Thus $K_{e^{u}}(\cdot)$ is a one-parameter subgroup. If there is another one-parameter subgroup $K(t):\mathbb{R}\rightarrow G$ satisfying $K(0)=e^{0}$ and $K(1)=u$, then for each $n\in\mathbb{Z}$ and $m\in \mathbb{P}$, we get $K(n)=e^{nu}$, and $K(\frac{n}{m})=e^{\frac{n}{m}u}$. By the continuity of $K$, for each $r\in\RR$, $K(r)=e^{ru}$ holds. Therefore, $K=K_{e^{u}}$. Thus $G$ is an $\mathbb{R}$-group.
\end{exam}

\begin{remark}
On an $\RR$-group $G$, take $\lambda\in {\bfk} \backslash \{0\}$ and $\mfrakL(a) := a^\lambda$, $\mfrakH(a) :=  a^{\frac{1}{\lambda}}$. Then Eq.~(\ref{eq:rbgw}) gives the notion of the {\bf Rota-Baxter operator with weight $\lambda$}:
\vpa
$$\frakB(a) \frakB(b) =  \frakB \Big(   \big( a^{\lambda} \frakB(a) b^{\lambda} \frakB(a)^{-1} \big)^{\frac{1}{\lambda}} \Big), \quad  a,b\in G,
\vpa
$$
a notion that first appeared formally in~\cite{BN}.
\label{exam:wtlambda}
\end{remark}

Now we give another example of of $(\mfrakL_{\lambda},\mfrakH_{\lambda})$-subsets.
\begin{exam}
Let $g$ be a $\mathbb{C}^{n\times n}$ matrix Lie algebra with exponential map $\exp$ bijective. Thus $\exp(g)$ an $\mathbb{R}$-group. Define $G:= \exp(\frakg) \cdot \{I,-I\}$, where $I$ is the identity matrix in $\mathbb{C}^{n\times n}$ and $\cdot$ is the multiplication of matrices. Define the set $S:=\exp(\frakg)\cdot (-I)$ and maps 
\vpa
$$
\mfrakL_{\lambda}:G\rightarrow G, a\mapsto\Bigg\{
\begin{array}{cc}
         -I \cdot P_{\lambda}(a),& \text{ if } a\notin S, \\
         P_{\lambda}(-I\cdot a),& \text{ if } a\in S,
\end{array}  
\quad \mfrakH_{\lambda}:G\rightarrow G, a\mapsto \Bigg\{ 
\begin{array}{cc}
      P_{\frac{1}{\lambda}}(-I \cdot a),   &  \text{ if } a \in S, \\
     -I \cdot P_{\frac{1}{\lambda}}(a),   &  \text{ if } a \notin S.
\end{array}  
\vpa
$$
For each $a$ in $S$, we have 
$$\mfrakL_{\lambda}\mfrakH_{\lambda}(a)=-I\cdot P_{\lambda}P_{\frac{1}{\lambda}}(-I\cdot a)=a\,\text{ and }\, \lim\limits_{n\to\infty}\mfrakL_{\frac{1}{n}}(a)=\lim\limits_{n\to\infty}P_{\frac{1}{n}}(-I\cdot a)=I.$$
Thus $S$ is a $(\mfrakL_{\lambda},\mfrakH_{\lambda})$-subset of $G$.
\label{ex:ljfz}
\end{exam}

As an example of Definition~\ref{defn:lhnrb}, we have the following concept by taking $\mfrakL_{\frac{1}{n}}:=\pown^{-1}$ and $\mfrakH_{\frac{1}{n}}:=\pown$.

\begin{defn}
\label{defn:rbg0}
Let $G$ be a $\mathbb{Q}$-group and also a $\lim\limits_{n\to\infty}(P_n^{-1},P_n)$-group.
A map $\frakB:G\rightarrow G$ is called a {\bf Rota-Baxter operator with limit-weight zero}
if for each $(a,b)\in G\times G$, $\lim\limits_{n\to \infty}  \Big( a^{\frac{1}{n}} \frakB(a) b^{\frac{1}{n}} \frakB(a)^{-1}\Big) ^n $ uniquely exists and
\vpa
\begin{equation}
\frakB(a) \frakB(b) =  \frakB \bigg(\lim_{n\to \infty}  \Big( a^{\frac{1}{n}} \frakB(a) b^{\frac{1}{n}} \frakB(a)^{-1}\Big) ^n \bigg).
\label{eq:rbo0}
\vpa
\end{equation}
Then we call $(G,\frakB)$ a {\bf Rota-Baxter $\mathbb{Q}$-group with limit-weight zero}. Further, if $G$ is a Lie group and $\frakB$ is a smooth map, then we call  $(G,\frakB)$ a {\bf Rota-Baxter Lie $\mathbb{Q}$-group with limit-weight zero}.
\end{defn}

The above notion of Rota-Baxter $\mathbb{Q}$-group with limit-weight zero is justified by the following relation with a Rota-Baxter Lie algebra with weight zero.
	
\begin{theorem}
Let $(G, \frakB)$ be a Rota-Baxter Lie $\mathbb{Q}$-group with limit-weight zero.
Let $\frakg = T_e G$ be the Lie algebra of $G$ and $B:\frakg \rightarrow \frakg$
the tangent map of $\frakB$ at the identity $e$. Then $(\frakg, B)$ is a Rota-Baxter Lie algebra with limit-weight $\lim\limits_{n\to\infty}(\frac{1}{n}\id_\frakg ,n\,\id_\frakg)$, and hence a Rota-Baxter Lie algebra with weight zero by Remark~\ref{rk:weight 0}.
\label{thm:rbg2rbl0}
\end{theorem}
	
\begin{proof}
We only need to verify
\vpa
$$\left [ B(u), B(v) \right ] = B\Big([B(u),v]+[u,B(v)]\Big), \quad u,v\in \frakg,
\vpa$$
Since $P_n$ and $P_n^{-1}$ are unital maps and, for each $u\in\frakg$, we have
\vpa
$$(\pown)_{*e}(u)=nu\,\text{ and }\,(\pown^{-1})_{*e}(u)=\frac{u}{n}.
\vpa$$
By Theorem~\ref{thm:tgop},
\vpb
\begin{align*}
[B(u),B(v)] &=  B\bigg(\lim_{n\to\infty}(\pown^{-1})_{*e}\Big([B(u),(\pown)_{*e}(v)]+[(\pown)_{*e}(u),B(v)]+[(\pown)_{*e}(u),(\pown)_{*e}(v)]\Big) \bigg)\\
    &= B\Big(\lim_{n\to\infty} [B(u),v]+[u,B(v)]+\frac{1}{n}[u,v]\Big)= B\Big( [B(u),v]+[u,B(v)]\Big). \qedhere
\end{align*}
\end{proof}

The following concept of the integral on an $\mathbb{R}$-group is a special case of Definition ~\ref{defn:IntLH}, by taking $\mfrakL_\lambda:a\mapsto P_{\lambda}(a) = a^\lambda$ and $\mfrakH_\lambda:a\mapsto P_{\frac{1}{\lambda}}(a) = a^{\frac{1}{\lambda}}$ for $\lambda\neq 0$.

\begin{defn}
Let $G$ be an $\mathbb{R}$-group and $a:\mathbb{R}\to G$ a map. If $\lim\limits_{|P|\to 0 }\prod\limits_{k=1}^{n}\Big(a(\xi_{n+1-k})\Big)^{\triangle_{n+1-k}}$ uniquely exists for arbitrary partitions $\Delta$ of $[\ell, m]$ and arbitrary choices of $\xi_i$ in $[x_i,x_{i-1}], 1\leq i\leq n$, then we say that $a(x)$ is {\bf $\mathbb{R}$-integrable} on $[\ell,m]$, denote the limit by $\int_{\ell}^{m}a(x)d_Gx$ and call it the {\bf $\mathbb{R}$-integral of the map $a$ on $[\ell,m]$}.
\label{defn:IntR}
\end{defn}

\begin{remark}
In the special case when the $\mathbb{R}$-group $G$ is $\mathbb{R}$ and the multiplication of $G$ is the addition of $\mathbb{R}$, then
\vpb
\[
\lim\limits_{|P|\to 0 }\prod\limits_{k=1}^{n}\Big(a(\xi_{n+1-k})\Big)^{\triangle_{n+1-k}} = \lim\limits_{|P|\to 0 }\sum\limits_{k=1}^{n}{\triangle_{n+1-k}}a(\xi_{n+1-k})
\vpa
\]
is simply the Riemann integral of a real-valued function on $\mathbb{R}$.
See Lemma~\ref{lemma:intbm} for further examples.
\label{rk:Riem}
 \end{remark}

The following is the integration-by-parts formula for $\mathbb{R}$-integrals.

\begin{coro}
\label{co:rgpint}
Let $G$ be an $\mathbb{R}$-group which is also a \compgroup group, and $a, b:[0,t]\to G$ be two $\mathbb{R}$-integrable maps for each $t\in [0,x]$.
Suppose that the map
\[
[0,x]\to G,\quad t\mapsto \lim\limits_{n\to \infty}\bigg(a(t)^{\frac{1}{n}}\Big(\int_{0}^{t}a(s)d_Gs\Big) b(t)^{\frac{1}{n}}\Big(\int_{0}^{t}a(s)d_Gs\Big)^{-1}\bigg)^{n}
\]
is \compint on $[0,x]$, as defined in Definition~\ref{defn:intsnow}~\eqref{it:int-commutative}. Then
    \begin{align*}
        \bigg(\int_{0}^{x}a(t)d_Gt\bigg)\bigg(\int_{0}^{x}b(t)d_Gt\bigg)=\int_{0}^{x}\lim_{n\to \infty}\bigg(a(t)^{\frac{1}{n}}\Big(\int_{0}^{t}a(s)d_Gs\Big) b(t)^{\frac{1}{n}}\Big(\int_{0}^{t}a(s)d_Gs\Big)^{-1}\bigg)^{n}d_Gt.
    \end{align*}
\end{coro}

\begin{proof}
It follows from Theorem ~\ref{thm:HIfml} by taking $\mfrakL_{\frac{1}{n}}:=\pown^{-1}$ and $\mfrakH_{\frac{1}{n}}:=\pown$.
\end{proof}

\subsection{Exponential Rota-Baxter groups with limit-weight zero}
In this subsection, we calculate Rota-Baxter operators with limit-weight zero in some special cases where the limit in Definition~\ref{defn:rbg0} can be eliminated.

Suppose that $G$ is a Lie group with a matrix Lie algebra $\frakg$ so that $\exp:\frakg \rightarrow G$ is bijective. Then we can write
\vpb
\begin{equation}
\label{eq:expA}
\exp(X)=\sum\limits_{k=1}^{\infty}\frac{X^{k}}{k!}, \quad  X\in \frakg.
\vpb
\end{equation}
Since the map $\pown:G \rightarrow G, \, a\mapsto a^{n}$ is bijective for each $n\in \mathbb{P}$, $G$ is a $\mathbb{Q}$-group.
Let $\frakB:G\rightarrow G$ be an operator and $X\in \frakg$.
There is a unique $Y_X\in g$ such that $\frakB(\exp(X)) = \exp(Y_X)$. Hence we obtain a map 
\vpc
\begin{equation}
\widetilde\frakB:\frakg\rightarrow\frakg, \quad X\mapsto Y_X,
\label{eq:tilfb}
\vpa
\end{equation}
called the {\bf \mcorr of $\frakB$ on $\frakg$}.

\begin{lemma}\cite[Proposition 3.35]{H1}
Let $X\in \mathbb{C}^{n\times n}$. Then
\vpa
\begin{equation}
\lim\limits_{n\to \infty}\bigg(1+\frac{X}{n}+o\Big(\frac{1}{n}\Big)\bigg)^{n}=\exp(X), \label{eq:expe}
\end{equation}
for the $\exp(X)$ given in Eq.~\eqref{eq:expA}.
\end{lemma}

As an immediate consequence, we have the following result.

\begin{coro}
\label{coro:0mult}
Let $G$ be a complex matrix Lie group with a complex matrix Lie algebra $\frakg$ and
$\exp:\frakg \rightarrow G$ bijective.
Then $G$ is a complete  $\lim\limits_{n\to \infty}(\pown^{-1},\pown)$-group with the multiplication given by
\begin{equation}
\label{eq:xinfy}
\exp(X)\cdot_{\frac{1}{\infty}}\exp(Y) :=\exp(X+Y), \quad \exp(X),\exp(Y)\in G.
\end{equation}
\end{coro}

\begin{proof}
We first have
\vpb
\begin{align*}
\lim_{n\to\infty}\exp(0)^{\frac{1}{n}}=&\ \exp(0), \quad \lim_{n\to\infty}\Big(\exp(nX)^{-1}\Big)^{\frac{1}{n}}=\ \exp(-X),\\
    \exp(X)\cdot_{\frac{1}{\infty}}\exp(Y)=&\ \lim\limits_{n\to \infty}\bigg(\exp\Big(\frac{X}{n}\Big)\exp\Big(\frac{Y}{n}\Big)\bigg)^{n}=\lim\limits_{n\to \infty}\bigg(\sum_{r,s=0}^{\infty}
    \frac{\big(\frac{X}{n}\big)^{r}\big(\frac{Y}{n}\big)^{s}}{r!s!}\bigg)^{n}\\
=&\ \lim\limits_{n\to \infty}\bigg(1+\frac{X+Y}{n}+o\Big(\frac{1}{n}\Big)\bigg)^n=\exp(X+Y).
\end{align*}
Here the second to the last step employs Eq.~(\ref{eq:expe}).

Now let $X_m,Y_m\in \frakg$ and $m\in \PP,$ with $\lim\limits_{m\to\infty}X_m=X$ and $\lim\limits_{m\to\infty}Y_m=Y$. So we have $X_m=X+\varepsilon_{m}^X$ and $Y_m=Y+\varepsilon_{m}^Y$ with $\lim\limits_{m\to\infty}\varepsilon_{m}^Y=\lim\limits_{m\to\infty}\varepsilon_{m}^X=0$. Then
\vpb
\begin{align*}
    \lim_{n\to\infty}\bigg(\exp\Big(\frac{X_n}{n}\Big)\exp\Big(\frac{Y_n}{n}\Big)\bigg)^{n}=&\ \lim\limits_{n\to \infty}\bigg(1+\frac{X+Y}{n}
    +\frac{\varepsilon_{n}(X)+\varepsilon_{n}(Y)}{n}+o\Big(\frac{1}{n}\Big)\bigg)^n\\
    =&\ \lim\limits_{n\to \infty}\bigg(1+\frac{X+Y}{n}+o\Big(\frac{1}{n}\Big)\bigg)^n=\exp(X+Y),
\end{align*}
showing that Eq.~\eqref{eq:jx} is satisfied. Thus $G$ is a complete  $\lim\limits_{n\to \infty}(\pown^{-1},\pown)$-group.
\end{proof}

Now we give an example built from functions on a nilpotent Lie group. Let
\begin{equation}
\begin{aligned}
\frakg :=&\ \{ (a_{ij}) \in \mathbb{C}^{3 \times 3} \mid a_{ij} = 0\,\text{ for }\,i\geqslant j\},\\
\frakJ:=&\ \{v:\mathbb{R} \rightarrow \frakg \text{ smooth}\}.
\end{aligned}
\label{eq:grj}
\end{equation}
Notice that $\frakg$ is a nilpotent matrix Lie algebra. Define
\begin{align*}
\exp:\frakJ\to\exp(\frakJ), \quad u\mapsto \exp(u),
\end{align*}
where $\exp(u)$ is defined pointwise by $\exp(u)(x):= \exp(u(x))$ for $x\in \RR$.

\begin{lemma}
The group $\exp(\frakJ)$ is a complete $\lim\limits_{n\to \infty}(\pown^{-1},\pown)$-group and a $\mathbb{R}$-group, where the topology is the product topology on $\mathbb{R}\times G$.
\label{Lemma:thklyf}
\end{lemma}

\begin{proof}
Since $\exp:\frakg \rightarrow \exp(\frakg)$ is bijective, the map $\exp$ has a left inverse, denoted by $\widetilde{\exp}$. For each $\exp(u)\in\exp(\frakJ)$, we have $u=\widetilde{\exp}\exp(u)$ and so $\exp:\frakJ\rightarrow \exp(\frakJ)$ is injective and then bijective.
For each $u,v\in \frakJ$ and $x\in\RR$, it follows from the nilpotency of $\frakg$ that
\begin{align*}
\exp(u(x))\exp(v(x))=&\ \exp\Big(u(x)+v(x)+\frac{1}{2}[u(x),v(x)]\Big),\\
\exp(0)\exp(u(x))=&\ \exp(u(x))\exp(x)=\exp(u(x)),\\
\exp(u(x))\exp(-u(x))=&\ \exp(0).
\end{align*}
Notice that $u+v+\frac{1}{2}[u,v],\,0,\,-u$ are in $\frakJ$, and  $\exp(\frakJ)$ is a group.
For each $r\in\mathbb{R}$, define the map
$$
P_{r}:\exp(\frakJ)\rightarrow \exp(\frakJ), \quad \exp(u) \mapsto \exp(ru),$$
where $\exp(ru)$ is defined pointwise by $\exp(ru)(x):= \exp(ru(x))$ for $x\in \RR$.

In terms of Corollary~\ref{coro:0mult}, for each sequences $u_n,v_n\in \frakJ, n\in \PP,$ with $\lim\limits_{n\to\infty}u_n=u,\lim\limits_{n\to\infty}v_n=v$ and $x\in\mathbb{R}$, we have
\vpb
\begin{align*}
\exp(u(x))\cdot_{\frac{1}{\infty}}\exp(v(x))=&\ \exp(u(x)+v(x)),\\
\exp(u(x)+v(x))=&\ \lim_{n\to\infty}\Big(\exp(\frac{1}{n}u_n(x))\exp(\frac{1}{n}v_n(x))\Big)^{n},
\end{align*}
showing that $\exp(\frakJ)$ is a complete $\lim\limits_{n\to\infty}(\pown^{-1},\pown)$-group.
\end{proof}

\begin{theorem}
\label{thm:rboze}
Let $G$ be a matrix Lie group with a matrix Lie algebra $\frakg$ such that
$\exp:\frakg \rightarrow G$ is bijective $($resp. Let $\frakJ$ be given in Eq.~\eqref{eq:grj}$)$.
Then a map  $\frakB:G \rightarrow G$ (resp. $\frakB^{\frakJ}: \exp(\frakJ) \rightarrow \exp(\frakJ)$) is a Rota-Baxter operator with $($resp. limit-$)$weight zero if and only if
\begin{equation*}
\frakB\big(\exp(u\big)\big)\frakB\big(\exp\big(v\big)\big)=
\frakB\Big(\exp\big(u+\frakB(\exp(u))v\frakB(\exp(u))^{-1}\big)\Big), \quad  u, v\in \frakg.
\label{eq:rboze}
\end{equation*}
\begin{equation*}
   \Big(\text{resp. } \frakB^{\frakJ}\big(\exp(u\big)\big)\frakB^{\frakJ}\big(\exp\big(v\big)\big)=
\frakB^{\frakJ}\Big(\exp\big(u+\frakB^{\frakJ}(\exp(u))v\frakB^{\frakJ}
(\exp(u))^{-1}\big)\Big), \quad u,v\in\frakJ \Big).
\end{equation*}
\end{theorem}

Note that the map $\frakB\circ \exp: \frakg \to G$ is the Rota-Baxter operator with weight zero in the sense of~\cite{LST2} (see Definition~3.3). See~\cite{GGH2} for relative Rota-Baxter operators on groups with limit-weights and their relation with pre-groups. 

\begin{proof}
We only consider the case of $\frakB:G \rightarrow G$.
Let $\frakB:G\rightarrow G$ be a map and $\widetilde\frakB$ be the \mcorr of $\frakB$ on $\frakg$. We have
{\small
\begin{align*}
&\ \lim_{n \to \infty}\bigg(\exp\Big(\frac{u}{n}\Big)\exp(\widetilde \frakB(u))\exp\Big(\frac{v}{n}\Big)\exp\Big(-\widetilde \frakB(u)\Big)\bigg)^{n}\\
=&\ \lim_{n \to \infty}\Bigg(\bigg(\sum_{r_1,s_1=0}^{\infty}\frac{(\frac{u}{n})^{r_1}\widetilde \frakB(u)^{s_1}}{r_1!s_1!}\bigg)
    \bigg(\sum_{r_2,s_2=0}^{\infty}\frac{(\frac{v}{n})^{r_2}(-\widetilde\frakB(u))^{s_2}}{r_2!s_2!}\bigg)\Bigg)^n\\
=&\ \lim_{n \to \infty}\Bigg(\bigg(I + \widetilde \frakB(u) + \frac{\widetilde \frakB(u)^2}{2!} +\cdots+\frac{u}{n}\Big(I + \widetilde \frakB(u) + \frac{\widetilde \frakB(u)^2 }{2!}+\cdots\Big)+o\Big(\frac{1}{n}\Big)\bigg)\\
&\hspace{1cm} \bigg(I - \widetilde \frakB(u) + \frac{\widetilde \frakB(u)^2 }{2!}+\cdots+\frac{v}{n}\Big(I - \widetilde \frakB(u) + \frac{\widetilde \frakB(u)^2}{2!} +\cdots\Big)+o\Big(\frac{1}{n}\Big)\bigg)\Bigg)^n\\
=   &\ \lim_{n \to \infty}\Bigg(\bigg(\exp(\widetilde \frakB(u))+\frac{u}{n}\exp(\widetilde \frakB(u))+o(\frac{1}{n})\bigg)\bigg(\exp(-\widetilde \frakB(u))+\frac{v}{n}\exp(-\widetilde \frakB(u))+o\Big(\frac{1}{n}\Big)\bigg)\Bigg)^n\\
=&\ \lim_{n \to \infty}\bigg(I+\frac{1}{n}\Big(u+\exp(\widetilde \frakB(u))v\exp(-\widetilde \frakB(u))\Big)+o\Big(\frac{1}{n}\Big)\bigg)^n\\
=&\ \exp\bigg(u+\exp(\widetilde \frakB(u))v\exp\Big(-\widetilde \frakB(u)\Big)\bigg)\quad \text{(by Eq.~(\ref{eq:expe}))}\\
=&\ \exp\bigg(u+\frakB\Big(\exp(u)\Big)v\frakB\Big(\exp(u)\Big)^{-1}\bigg).
\end{align*}
}
Now if $\frakB$ is a Rota-Baxter operator with limit-weight zreo, then by Eq.~(\ref{eq:rbo0}) we obtain
{\small
\begin{align*}
\frakB\big(\exp(u\big)\big)\frakB\big(\exp\big(v\big)\big)&=\frakB\Bigg( \lim_{n \to \infty}\bigg(\exp\Big(\frac{u}{n}\Big) \frakB(\exp(u))\exp\Big(\frac{v}{n}\Big)\frakB(\exp(u))^{-1}\bigg)^{n} \Bigg)\\
&=\frakB\Bigg( \lim_{n \to \infty}\bigg(\exp\Big(\frac{u}{n}\Big)\exp(\widetilde \frakB(u))\exp\Big(\frac{v}{n}\Big)\exp(-\widetilde \frakB(u))\bigg)^{n} \Bigg)\\
&=\frakB\bigg(\exp\Big(u+\frakB(\exp(u))v\frakB(\exp(u))^{-1}\Big)\bigg).
\end{align*}
}
Conversely, if
$$\frakB\big(\exp(u\big)\big)\frakB\big(\exp\big(v\big)\big)
=\frakB\bigg(\exp\Big(u+\frakB(\exp(u))v\frakB(\exp(u))^{-1}\Big)\bigg),$$
then it follows from Eq.~(\ref{eq:rbo0}) that
{\small
\begin{align*}
    \frakB\big(\exp(u\big)\big)\frakB\big(\exp\big(v\big)\big)
&=\frakB\bigg(\exp\Big(u+\frakB(\exp(u))v\frakB(\exp(u))^{-1}\Big)\bigg)\\
&=\frakB\Bigg( \lim_{n \to \infty}\bigg(\exp\Big(\frac{u}{n}\Big)\exp(\widetilde \frakB(u))\exp\Big(\frac{v}{n}\Big)\exp(-\widetilde \frakB(u))\bigg)^{n} \Bigg)\\
&=\frakB\Bigg( \lim_{n \to \infty}\bigg(\exp\Big(\frac{u}{n}\Big) \frakB(\exp(u))\exp\Big(\frac{v}{n}\Big)\frakB(\exp(u))^{-1}\bigg)^{n} \Bigg).
\qedhere
\end{align*}
}
\end{proof}

\begin{remark}
In terms of the map $\frakB\Big(\exp(\cdot)\Big):\frakg \rightarrow G$, we obtain the Rota-Baxter operator with weight $0$ on Lie group $G$ given in~\cite{LST2}.
\end{remark}

Let us expose an example of Rota-Baxter group with limit-weight $0$, which needs the following fact.

\begin{lemma}\cite[Proposition 3.35]{H1}.
Let $X,Y\in \mathbb{C}^{n\times n}$. Then
\vpa
\begin{equation}
\label{eq:BHF}        
\exp(X)Y\exp(-X)=Y+\sum\limits_{k=1}^{\infty}\frac{1}{k!} \ada_X^{k}Y, \quad \text{where\ \   } \ada_X Y:= [X, Y]. 
\vpa
\end{equation}
\end{lemma}
Now we calculate the integral of functions taking values in a nilpotent group.
\begin{lemma}
Let $\frakJ$ be the group given in Eq.~\eqref{eq:grj}. Then for each $u\in\frakJ$, we have
\begin{align*}
     \int_{0}^{x}u(t)d_Gt=\exp\Big(\int_{0}^xu(t)dt+\frac{1}{2}\int_{0}^x[u(t),\int_{0}^tu(s)ds]dt\Big)
\end{align*}
and $\int_{0}^{x}(\cdot)  d_Gt$ is a map from $\frakJ$ to $\exp(\frakJ)$.
\label{lemma:intbm}
\end{lemma}

\begin{proof}
We have
{\small
\begin{align*}
 \int_{0}^{x}u(t)d_Gt=&\ \lim_{n\to\infty}\prod_{k=1}^{n}\exp\bigg(\frac{x}{n}u\Big(\frac{n+1-k}{n}x\Big)\bigg)\\
=&\ \lim_{n\to \infty}\exp\bigg(\sum_{k=1}^{n}\frac{x}{n}u\Big(\frac{n+1-k}{n}x\Big)+\sum_{1\leqslant i< j\leqslant 1}\frac{1}{2}\Big(\frac{x}{n}\Big)^2\Big[u\Big(\frac{n+1-i}{n}x\Big),u\Big(\frac{n+1-j}{n}x\Big)\Big]\bigg)\\
&\hspace{5cm}\text{(by the nilpotency of $\frakg$)}\\
        =&\ \lim_{n\to \infty}\exp\bigg(\sum_{k=1}^{n}\frac{x}{n}u\Big(\frac{n+1-k}{n}x\Big)+\sum_{1\leqslant j< i\leqslant n}\frac{1}{2}\Big(\frac{x}{n}\Big)^2\Big[u\Big(\frac{i}{n}x\Big),u\Big(\frac{j}{n}x\Big)\Big]\bigg)\\
        =&\ \lim_{n\to \infty}\exp\bigg(\sum_{k=1}^{n}\frac{x}{n}u(\frac{n+1-k}{n}x)+\sum_{i=1}^n\sum_{j=1}^i\frac{1}{2}(\frac{x}{n})^2\Big[u\Big(\frac{i}{n}x\Big),u\Big(\frac{j}{n}x\Big)\Big]\bigg).
    \end{align*}
}
Notice that $$\lim\limits_{n\to\infty}\sum_{k=1}^{n}\frac{x}{n}u\Big(\frac{n+1-k}{n}x\Big)=\int_{0}^xu(t)dt$$ and
\begin{align*}
\lim\limits_{n\to\infty}\sum\limits_{i=1}^n
\sum\limits_{j=1}^i\Big(\frac{x}{n}\Big)^2\Big[u\Big(\frac{i}{n}x\Big),u\Big(\frac{j}{n}x\Big)\Big]
=\ \lim\limits_{n\to\infty}\sum\limits_{i=1}^n
\sum\limits_{j=1}^i\Big[u\Big(\frac{i}{n}x\Big),\frac{x}{n}u\Big(\frac{j}{n}x\Big)\Big]\frac{x}{n}
        =\ \int_{0}^x[u(t),\int_{0}^tu(s)ds]dt.
    \end{align*}
So we obtain
\begin{align*}
\int_{0}^{x}u(t)d_Gt =&\ \lim_{n\to \infty}\exp\bigg(\sum_{k=1}^{n}\frac{x}{n}u\Big(\frac{n+1-k}{n}x\Big)+\sum_{i=1}^n\sum_{j=1}^i\frac{1}{2}\Big(\frac{x}{n}\Big)^2\Big[u\Big(\frac{i}{n}x\Big),u\Big(\frac{j}{n}x\Big)\Big]\bigg)\\
=&\ \exp\Big(\int_{0}^xu(t)dt+\frac{1}{2}\int_{0}^x[u(t),\int_{0}^tu(s)ds]dt\Big).
\end{align*}
Since
$$\int_{0}^xu(t)dt+\frac{1}{2}\int_{0}^x[u(t),\int_{0}^tu(s)ds]dt$$
is in $\frakJ,$ we conclude that $\int_{0}^{x}(\cdot) d_Gt$ is a map with values in $\exp(\frakJ)$.
\end{proof}

Now we give an example of a Rota-Baxter group with limit-weight zero.

\begin{prop}
Let $\frakg=\{ (a_{ij}) \in \mathbb{C}^{3 \times 3} \mid a_{ij} = 0\,\text{ for }\,i\geqslant j\}$ be a nilpotent matrix Lie algebra with a Rota-Baxter operator $B$ with weight zero on $\frakg$ satisfying $\frakB([\frakg,\frakg])\subseteq [\frakg,\frakg]$ and $G:=\exp(\frakg)$ be the simply connected nilpotent analytic Lie group whose tangent space is $\frakg$.
Define a map
\begin{equation*}
\frakB:G\rightarrow G, \quad \exp(u)\mapsto \exp\bigg(B(u)+\frac{1}{2}B\Big([u,B(u)]\Big)\bigg).
\label{eq:exafb}
\end{equation*}
\begin{equation*}
\text{\bigg(resp.\ }\int_{0}^{x}(\cdot) d_Gt:\exp(\frakJ)\rightarrow \exp(\frakJ), \quad \exp(u(x))\mapsto \exp\Big(\int_{0}^xu(t)dt+\frac{1}{2}\int_{0}^x[u(t),\int_{0}^tu(s)ds]dt\Big)\text{\bigg ).}
\end{equation*}
Then $(G,\frakB)$ $($resp. $(\frakJ,\int_{0}^{x}(\cdot) d_Gt)$$)$ is a Rota-Baxter group with limit-weight zero.
\label{prop:rb33}
\end{prop}

\begin{proof}
We restrict ourselves to the case $\frakB:G\rightarrow G$.
Owing to Corollary~\ref{coro:0mult},
$G$ is a complete $(\pown^{-1},\pown)$-group. By the nilpotency of $\frakg$ and Baker-Campbell-Hausdorff formula, on the one hand,
\begin{align*}
\frakB(\exp(u))\frakB(\exp(v))=&\ \exp\Big(B(u)+\frac{1}{2}B[u,B(u)]\Big)\exp\Big(B(v)+\frac{1}{2}B[v,B(v)]\Big)\\
        =&\ \exp\Big(B(u)+B(v)+\frac{1}{2}B[u,B(u)]+\frac{1}{2}B[v,B(v)]+\frac{1}{2}[B(u),B(v)]\Big)\\
        =&\ \exp\Big(B(u)+B(v)+\frac{1}{2}B[u,B(u)]+\frac{1}{2}B[v,B(v)]+\frac{1}{2}[B(u),v]+\frac{1}{2}[u,B(v)]\Big).
    \end{align*}
On the other hand,
{\small
\begin{align*}
&\frakB(\exp(u+\frakB(\exp(u))v\frakB(\exp(u))^{-1}))\\
=&\ \frakB\Bigg(\exp\bigg(u+\exp\Big(B(u)+\frac{1}{2}B[u,B(u)]\Big)v\exp
\Big(-B(u)-\frac{1}{2}B([u,B(u)])\Big)\bigg)\Bigg)\\
=&\ \frakB\Bigg(\exp\bigg(u+v+\sum_{k=1}^{\infty}
 \frac{1}{k!}{\rm ad}_{B(u)+\frac{1}{2}B([u,B(u)])}^kv\bigg)\Bigg)
\quad\text{(by Eq.~(\ref{eq:BHF})) }\\
=&\ \frakB\Bigg(\exp\bigg(u+v+[B(u),v]\bigg)\Bigg)\quad \text{(by the nilpotency of $\frakg$)}
\\
=&\ \exp\Bigg(B(u)+B(v)+B([B(u),v])+\frac{1}{2}B\bigg(\Big[u+v+[B(u),v],B(u+v+[B(u),v]\Big]\bigg)\Bigg)\\
=&\ \exp\Bigg(B(u)+B(v)+B([B(u),v])+\frac{1}{2}B\bigg(\Big[u+v+[B(u),v],B(u)+B(v)+B([B(u),v]\Big]\bigg)\Bigg)\\
=&\ \exp\Bigg(B(u)+B(v)+B([B(u),v])+\frac{1}{2}B\bigg([u,B(u)]+[u,B(v)]+[v,B(u)]+[v,B(v)]\bigg)\Bigg)\\
=&\ \exp\Big(B(u)+B(v)+\frac{1}{2}B[u,B(u)]+\frac{1}{2}B[v,B(v)]+\frac{1}{2}[B(u),v]+\frac{1}{2}[u,B(v)]\Big).
\end{align*}
}
Thus the proof is completed by Theorem~\ref{thm:rboze}.
\end{proof}

\subsection{Differential $\QQ$-groups with weight zero}
In this subsection, we consider a special case of differential groups
with limit-weight zero, whose tangent space is a differential Lie algebra with weight zero.

\begin{defn}
Let $G$ be a $\mathbb{Q}$-group which is also a $\lim\limits_{n\to\infty}(P_n^{-1},P_n)$-group. A map $\mathfrak{D}:G\rightarrow G$ is called a {\bf differential operator with limit-weight zero} on $G$ if, for $a,b\in G$, the limit $\lim\limits_{n\to \infty} \Big(\frakD(a)^{\frac{1}{n}}a \frakD(b)^{\frac{1}{n}} a^{-1} \Big)^n$ uniquely exists and
\vpb
\begin{equation}	
\label{eq:diffg0}
\frakD\left(ab\right)=
\lim_{n\to \infty} \Big(\frakD(a)^{\frac{1}{n}}a \frakD(b)^{\frac{1}{n}} a^{-1} \Big)^n.
\vpa
\end{equation}
In this case, we call $(G,\frakD)$ a {\bf differential $\mathbb{Q}$-group with limit-weight zero}. Further, if $G$ is a Lie group and $\frakD$ is a smooth map, then we call  $(G,\frakD)$ a {\bf differential Lie $\mathbb{Q}$-group with limit-weight zero}.
\label{defn:diffg0}
\end{defn}

The following result captures the relation between differential Lie $\mathbb{Q}$-group with limit-weight zero and differential Lie algebra with weight zero.

\begin{theorem}
Let $(G, \frakD)$ be a differential Lie $\QQ$-group with limit-weight zero.
Let $\frakg = T_e G$ be the Lie algebra of $G$ and $D:\frakg \rightarrow \frakg$
the tangent map of $\frakB$ at the identity $e$.
Then $\left(\mathfrak{g},D\right)$ is a differential Lie algebra with limit-weight $(\lim\limits_{n\to\infty}\frac{1}{n}\id_\frakg, n\,\id_\frakg)$, and so a differential Lie algebra with weight zero.
\label{thm:dgpl0}
\end{theorem}

\begin{proof}
It suffices to prove 
\vpb
$$D\left[u,v\right] = \left[D\left(u\right),v\right] +\left[u,D\left(v\right)\right],\quad  u,v\in \frakg,$$
which follows from
\vpb
\begin{align*}
D\left[u,v\right]  =&\ \left.\dfrac{\dd^2}{\dd t\dd s}\right|_{t,s=0}\frakD\left(e^{tu}e^{sv}e^{-tu}\right)\\
=&\ \lim_{n \to \infty}\dfrac{\dd^2}{\dd t\dd s} \left( \frakD(e^{tu})^{\frac{1}{n}}e^{tu}
\frakD(e^{sv})^{\frac{1}{n}}e^{-tu}
e^{tu}e^{sv}\frakD\left(e^{-tu}\right)^{\frac{1}{n}}e^{-sv}e^{-tu}
\right)^{n} \\
=&\ \lim_{n \to \infty}\bigg( \left[D\left(u\right),v\right] +\left[u,D\left(v\right)\right] +\dfrac{1}{n}\left[D\left(u\right),D\left(v\right)\right] \bigg)\\
=&\ \left[D\left(u\right),v\right] +\left[u,D\left(v\right)\right].
\qedhere
\end{align*}
\end{proof}

In reference to Definition ~\ref{defn:DfLH}, we define derivatives of functions with values in $\mathbb{R}$-groups.

\begin{defn}
Let $G$ be an $\mathbb{R}$-group. A map $a:\mathbb{R}\to G$ is called {\bf $\RR$-(multiplicatively) differentiable} at $x$ if $\lim\limits_{\lambda \to0 }(a(x+\lambda)a(x)^{-1})^{\frac{1}{\lambda}}$  uniquely exists. In this case, we denote the limit by $\frac{d_G}{d_Gx}\Big(a(x)\Big)$ and call it the {\bf \bf$\mathbb{R}$-derivative} of $a$ at $x$.
\label{defn:DfR}
\end{defn}

\begin{remark}
When the $\mathbb{R}$-group $G$ is $\mathbb{R}$ and the multiplication of $G$ is the addition of $\mathbb{R}$, then $$\lim\limits_{\lambda \to\infty }(a(x+\lambda)a(x)^{-1})^{\frac{1}{\lambda}} = \lim\limits_{\lambda \to\infty }\frac{a(x+\lambda)-a(x)}{\lambda},$$
which is just the usual derivative in analysis. See Lemma~\ref{lemma:tsmj} for another example.
\end{remark}

The following is the Leibniz rule for $\mathbb{R}$-groups.

\begin{theorem}
Let $G$ be an $\mathbb{R}$-group and let $a:\mathbb{R}\to G$ be $\mathbb{R}$-differentiable at each $x\in\RR$. If
\begin{align*}
F_G: G\times G \rightarrow G,\quad  (c,d) \mapsto  \lim_{n\to \infty}\Big((c)^{\frac{1}{n}}a(x)(d)^{\frac{1}{n}}a(x)^{-1}\Big)^{n}
\end{align*}
is a \complim map for each $x\in \mathbb{R}$,
then
\begin{align*}
        \frac{d_G}{d_Gx}\Big(a(x)b(x)\Big)=\lim_{n\to \infty}\Bigg(\bigg(\frac{d_G}{d_Gx}\Big(a(x)\Big)\bigg)^{\frac{1}{n}}a(x)\bigg(\frac{d_G}{d_Gx}\Big(b(x)\Big)\bigg)^{\frac{1}{n}}a(x)^{-1}\Bigg)^{n}.
    \end{align*}
\end{theorem}

\begin{proof}
It follows from Theorem ~\ref{thm:HDfml} by
taking $\mfrakL_{\frac{1}{n}}:=\pown^{-1}$ and $
\mfrakH_{\frac{1}{n}}:=\pown$.
\end{proof}

\begin{theorem}
Let $G$ be a complex matrix Lie group and $\frakg$ the Lie algebra of $G$ at $e$ $($resp. Let $\frakJ$ be given in Eq.~\eqref{eq:grj}$)$. Suppose that the map $\exp:\frakg \rightarrow G$  is bijective. Then an operator $\frakD: G\rightarrow G$ $($resp. $\frakD^{\frakJ}:\exp(\frakJ)\rightarrow \exp(\frakJ)$$)$  is a differential operator with $($limit-$)$weight zero if and only if
\vpb
\begin{align}
    \frakD\Big(\exp(u\big)\exp\big(v\big)\Big)= \exp\Big(\widetilde \frakD(u)+\exp(u)\widetilde \frakD(v) \exp(u)^{-1}\Big), \quad u,v\in \frakg,
    \label{eq:diffg0e}
\end{align}
\begin{align*}
    \Big(\text{resp.\ }\frakD^{\frakJ}\Big(\exp(u\big)\exp\big(v\big)\Big)= \exp\Big(\widetilde{\frakD^{\frakJ}}(u)+\exp(u)\widetilde{\frakD^{\frakJ}}(v) \exp(u)^{-1}\Big), \quad u,v\in \frakJ\Big),
\end{align*}
where $\widetilde\frakD$ is the \mcorr of $\frakD$ on $\frakg$ $($resp. $\frakJ$$)$ given in Eq.~\eqref{eq:tilfb}.
\end{theorem}

\begin{proof}
We only prove the case of $\frakD: G\rightarrow G$, as the other case is similar.
We have
{\small
\begin{align*}
   &\ \lim_{n \to \infty}\Big(\exp(\frac{1}{n}\widetilde \frakD(u))\exp(u)\exp(\frac{1}{n}\widetilde \frakD(v))\exp(-u)\Big)^{n}\\
    =&\ \lim_{n \to \infty}\bigg(\Big(\sum_{r_1,s_1=0}^{\infty}\frac{(\frac{1}{n}\widetilde \frakD(u))^{r_1}u^{s_1}}{r_1!s_1!}\Big)
    \Big(\sum_{r_2,s_2=0}^{\infty}\frac{(\frac{1}{n}\widetilde \frakD(v))^{r_2}(-u)^{s_2}}{r_2!s_2!}\Big)\bigg)^n\\
    =&\ \lim_{n \to \infty}\bigg(\Big(I + u + \frac{u^2}{2!} +\cdots+\frac{1}{n}\widetilde \frakD(u)(I + u + \frac{u^2}{2!} +\cdots)+o(\frac{1}{n})\Big)\\
    &\hspace{1.3cm}\Big(I - u + \frac{u^2}{2!} +\cdots+\frac{1}{n}\widetilde \frakD(v)(I - u +\frac{u^2}{2!} +\cdots)+o(\frac{1}{n})\Big)\bigg)^n\\
    =&\ \lim_{n \to \infty}\bigg(\Big(\exp(u)+\frac{1}{n}\widetilde \frakD(u)\exp(u)+o(\frac{1}{n})\big)\big(\exp(-u)+\frac{1}{n}\widetilde \frakD(v)\exp(u)+o(\frac{1}{n})\Big)\bigg)^n\\
    =&\ \lim_{n \to \infty}\Big(I+\frac{1}{n}(\widetilde \frakD(u)+\exp(u)\widetilde \frakD(v)\exp(-u)+o(\frac{1}{n})\Big)^n\\
    =&\ \exp\Big(\widetilde \frakD(u)+\exp(u)\widetilde \frakD(v)\exp(-u)\Big)
= \exp\Big(\widetilde \frakD(u)+\exp(u)\widetilde \frakD(v) \exp(u)^{-1}\Big).
\end{align*}
}
By Eq.~(\ref{eq:diffg0}), $\frakD:G\rightarrow G$ is a differential operator with weight zero if and only if
\begin{equation*}
 \frakD\Big(\exp(u\big)\exp\big(v\big)\Big) = \lim_{n \to \infty}\Big(\exp(\frac{1}{n}\widetilde \frakD(u))\exp(u)\exp(\frac{1}{n}\widetilde \frakD(v))\exp(-u)\Big)^{n},
\end{equation*}
which holds if and only if Eq.~(\ref{eq:diffg0e}) is valid.
\end{proof}

Now we calculate the derivation of functions on a nilpotent group.

\begin{lemma}
\label{lemma:tsmj}
Let $\frakJ$ be the group given in Eq.~(\ref{eq:grj}). Then for each $u\in\frakJ$,
\begin{align*}
     \frac{d_G}{d_Gx}\Big(\exp(u(x))\Big)=\exp\bigg(\frac{d}{dx}\Big(u(x)\Big)+\frac{1}{2}\Big[u(x),\frac{d}{dx}\Big(u(x)\Big)\Big]\bigg),
\end{align*}
and $\frac{d_G}{d_Gx}$ is a map on $\exp(\frakJ)$.
\end{lemma}

\begin{proof}
The first statement follows from
\vpb
{\small
\begin{align*}
\frac{d_G}{d_Gx}\Big(\exp(u(x))\Big)
=&\lim_{t\to 0}\bigg(\exp\Big(u(x+t)\Big)\exp\Big(-u(x)\Big)\bigg)^{\frac{1}{t}} \quad
        \text{(by Definition~\ref{defn:DfR})}\\
        =&\ \lim_{t\to 0}\exp\bigg(\frac{1}{t}\Big(u(x+t)-u(x)+\frac{1}{2}[u(x+t),-u(x)]\Big)\bigg)\quad
        \text{(by the nilpotency of $\frakg$)}\\
        =&\ \lim_{t\to 0}\exp\bigg(\frac{1}{t}\Big(u(x+t)-u(x)+\frac{1}{2}[u(x+t)-u(x),-u(x)]+\frac{1}{2}[u(x),-u(x)]\Big)\bigg)\\
        =&\ \lim_{t\to 0}\exp\bigg(\frac{u(x+t)-u(x)}{t}+\frac{1}{2}\Big[\frac{u(x+t)-u(x)}{t},-u(x)\Big]\bigg)\\
        =&\ \exp\bigg(\frac{d}{dx}\Big(u(x)\Big)+\frac{1}{2}\Big[u(x),\frac{d}{dx}\Big(u(x)\Big)\Big]\bigg).
    \end{align*}
}
Since $\frac{d}{dx}\Big(u(x)\Big)+\frac{1}{2}\Big[u(x),\frac{d}{dx}\Big(u(x)\Big)\Big]$ is in $\frakJ$ and $\exp$ is bijective, the operator $\frac{d_G}{d_Gx}$ is a map on $\exp(\frakJ)$. 
\end{proof}

Now we give an example of differential groups with limit-weight zero.
\begin{prop}
Let $\frakg=\{ (a_{ij}) \in \mathbb{C}^{3 \times 3} \mid a_{ij} = 0\, \text{ for }\,i\geqslant j\}$ be the nilpotent Lie algebra and $G=\exp(\frakg)$  be the simply connected nilpotent analytic Lie group whose tangent space is $\frakg$. Let $D$ be a differential operator on $\frakg$ with weighted zero. Define
\vpb
\begin{equation*}
\frakD:G\rightarrow G,\quad \exp(u)\mapsto \exp\bigg(D(u)+\frac{1}{2}[u,D(u)]\bigg),\quad u\in \frakg.
\label{eq:expdf}
\end{equation*}
\vpb
\begin{equation*}
   \text{\bigg(resp.\  } \frac{d_G}{d_Gx}:G\rightarrow G,\quad \exp(u)\mapsto \exp\bigg(\frac{d}{dx}(u(x))+\frac{1}{2}[u(x),\frac{d}{dx}(u(x))]\bigg),\quad u\in \frakJ \text{\bigg).}
\end{equation*}
Then $(G,\frakD)$ $($resp. $(\exp(\frakJ,\frac{d_G}{d_Gx}))$$)$ is a differential group with limit-weight zero.
\label{ex:diffgpzero}
\end{prop}

\begin{proof}
We only consider the case of $(G,\frakD)$.
First,
\vpb
{\small
\begin{align*}
\exp(u)\widetilde{\frakD}(v)\exp(-u)=&\ \bigg(\Big(\sum_{k=0}^{\infty}\frac{u^k}{k!}\Big)\widetilde{\frakD}(v)\Big(\sum_{k=0}^{\infty}\frac{(-u)^k}{k!}\Big)\bigg)\\
=&\ \widetilde{\frakD}(v)+\sum_{k=1}^{\infty} \frac{1}{k!}\ada_u^k\widetilde{\frakD}(v)\quad \text{(by Eq.~(\ref{eq:BHF}))}\\
=&\ \widetilde{\frakD}(v)+[u,\widetilde{\frakD}(v)]\quad\text{(by the nilpotency of $\frakg$)}\\
=&\ D(v)+[u,D(v)]+\frac{1}{2}[v,D(v)].
\vpb
\end{align*}
}
Next, applying the Baker-Campbell-Hausdorff formula and the nilpotency of $\frakg$,  we have 
\vpa
$$\exp(u)\exp(v)=\exp\Big(u+v+\frac{[u,v]}{2}\Big).
\vpb
$$
Whence
\vpa
{\small
\begin{align*}
&\frakD\bigg(\exp\Big(u\Big)\exp\Big(v\Big)\bigg)= \frakD\bigg(\exp\Big(u+v+\frac{[u,v]}{2}\Big)\bigg)\quad \text{(by the nilpotency of $\frakg$)}\\
=&\ \exp\bigg(D(u)+D(v)+D\Big(\frac{[u,v]}{2}\Big)+\frac{1}{2}\Big[u+v+\frac{[u,v]}{2},D(u)+D(v)+D\Big(\frac{[u,v]}{2}\Big)\Big]\bigg)\\
=&\ \exp\bigg(D(u)+D(v)+\frac{[D(u),v]}{2}+\frac{[u,D(v)]}{2}+\frac{1}{2}\Big[u+v+\frac{[u,v]}{2},D(u)+D(v)+\frac{[D(u),v]}{2}+\frac{[u,D(v)]}{2}\Big]\bigg)\\
=&\ \exp\Big(D(u)+D(v)+\frac{[D(u),v]}{2}+\frac{[u,D(v)]}{2}+\frac{1}{2}[u,D(u)]+\frac{1}{2}[v,D(u)]+\frac{1}{2}[u,D(v)]+\frac{1}{2}[u,D(u)]\Big)\\
=&\ \exp\Big(D(u)+\frac{1}{2}[u,D(v)]+D(v)+[u,D(v)]+\frac{1}{2}[v,D(v)]\Big)\\
=&\ \exp\Big(\widetilde{\frakD}(u)+\exp(u)\widetilde{\frakD}(v)\exp(-u)\Big). \qedhere
\end{align*}
}
\end{proof}
Now we give the First Fundamental Theorem of Calculus (Newton-Leibniz Formula) of functions on the special nilpotent Lie group $\exp(\frakg)$, where $\frakg$ is given in Eq.~(\ref{eq:grj}).

\begin{prop}
\label{pp:fftc}
Let $\exp(\frakJ)$ be the group where $\frakJ$ is given in Eq.~\eqref{eq:grj}. Then for each $u\in\frakJ$,
\begin{align*}
    \frac{d_G}{d_Gx}\Big(\int_{0}^{x}\exp(u(x)) d_Gt\Big)=\int_{0}^{x}\Big(\frac{d_G}{d_Gt}\exp(u(t))\Big) d_Gt=\exp(u(x)).
\end{align*}
\end{prop}
\begin{proof}
According to Lemmas~\ref{lemma:intbm} and~\ref{lemma:tsmj}, on the one hand,
{\small
\begin{align*}
&\ \frac{d_G}{d_Gx}\Big(\int_{0}^{x}\exp(u(x)) d_Gt\Big)= \frac{d_G}{d_Gx}\exp\bigg(\int_{0}^xu(t)dt+\frac{1}{2}\int_{0}^x[u(t),\int_{0}^tu(s)ds]dt\bigg)\\
=&\  \exp\Bigg(\frac{d}{dx}\bigg(\int_{0}^xu(t)dt+\frac{1}{2}\int_{0}^x\bigg[u(t),\int_{0}^tu(s)ds\bigg]dt\bigg)\\
&\ + \frac{1}{2}\bigg[\bigg(\int_{0}^xu(t)dt+\frac{1}{2}\int_{0}^x[u(t),\int_{0}^tu(s)ds]dt\bigg),\bigg(\frac{d}{dx}\Big(\int_{0}^xu(t)dt+\frac{1}{2}\int_{0}^x[u(t),\int_{0}^tu(s)ds]dt\Big)\bigg)\bigg]\Bigg)\\
=&\  \exp\Bigg(u(x)+\frac{1}{2}\bigg[u(x),\int_{0}^xu(s)ds\bigg]\\
&\ + \frac{1}{2}\bigg[\Big(\int_{0}^xu(t)dt+\frac{1}{2}\int_{0}^x[u(t),\int_{0}^tu(s)ds]dt\Big),\Big(u(x)+\frac{1}{2}[u(x),\int_{0}^tu(s)ds]\Big)\bigg]\Bigg)\\
=&\ \exp\Bigg(u(x)+\frac{1}{2}\bigg[u(x),\int_{0}^xu(s)ds\bigg] + \frac{1}{2}\bigg[\int_{0}^xu(t)dt,u(x)\bigg]\Bigg)\quad \text{(by the nilpotency of $\frakJ$) }\\
=&\ \exp(u(x)).
\end{align*}
}
On the other hand,
{\small
\begin{align*}
&\ \int_{0}^{x}\Big(\frac{d_G}{d_Gt}\exp(u(t))\Big) d_Gt=\int_{0}^{x}\Bigg(\exp\bigg(\frac{d}{dt}(u(t))+\frac{1}{2}[u(t),\frac{d}{dt}(u(t))]\bigg)\Bigg) d_Gt\\
=&\  \exp\Bigg(\int_{0}^x\bigg(\frac{d}{dt}(u(t))+\frac{1}{2}\Big[u(t),\frac{d}{dt}(u(t))\Big]\bigg)dt\\
&\ +\frac{1}{2}\int_{0}^x\bigg[\bigg(\frac{d}{dt}(u(t))+\frac{1}{2}[u(t),\frac{d}{dt}(u(t))]\bigg),\int_{0}^t\bigg(\frac{d}{ds}(u(s))+\frac{1}{2}[u(s),\frac{d}{ds}(u(s))]\bigg)ds\bigg]dt\Bigg)\\
=&\  \exp\Bigg(u(t)+\int_{0}^x\frac{1}{2}\Big[u(t),\frac{d}{dt}(u(t))\Big]dt+\int_{0}^x\frac{1}{2}\Big[\frac{d}{dt}(u(t)),u(t)]dt \Bigg)
\quad \text{(by the nilpotency of $\frakJ$) }\\
=&\ \exp(u(x)).
\end{align*}
}
Hence the proof is completed.
\end{proof}

\vpb

\noindent
{\bf Acknowledgments.} This work is supported by the National Natural Science Foundation of China (12071191) and Innovative Fundamental Research Group Project of Gansu Province (23JRRA684). The authors thank Chengming Bai and Yong Liu for helpful discussions, and Yifei Li for valuable programming.

\noindent
{\bf Declaration of interests.} The authors have no conflicts of interest to disclose.

\noindent
{\bf Data availability.} Data sharing is not applicable as no new data were created or analyzed.

\vpb

\end{document}